\documentclass[12pt]{article}

\usepackage{amsmath}
\usepackage{amsthm,amscd,amsfonts}
\usepackage{amssymb, upref, color}
\usepackage{amsmath,amssymb,amsthm,mathrsfs}
\usepackage{color}
\usepackage[dvips]{graphicx}
\usepackage{epsf,epsfig, verbatim} %subfigure,
\usepackage{latexsym,bm}

\numberwithin{equation}{section} %公式按节排号，这里去掉

\theoremstyle{plain}
\newtheorem{exam}{Example}[section]
\newtheorem{theorem}[exam]{Theorem}
\newtheorem{lemma}[exam]{Lemma}
\newtheorem{remark}[exam]{Remark}

\newtheorem{definition}[exam]{Definition}
\newtheorem{corollary}[exam]{Corollary}

\begin{document}
\title%[A generalization of Palmer's linearization theorem for nonautonomous differential equations]
{%A generalization of Palmer's linearization theorem for nonautonomous differential equations
%Higher regularity of homomorphisms in the Palmer's linearization theorem
Higher regularity of homeomorphisms in the Hartman-Grobman theorem and a conjecture on its sharpness
\footnote{ This paper was jointly supported from the National Natural
Science Foundation of China under Grant (No. 11931016 and 11671176) and  Grant Fondecyt 1170466.}
}
\author
{
Weijie Lu$^{a}$\,\,\,\,\,
Manuel Pinto$^{b}$\,\, \,\,
Y-H Xia$^{a}\footnote{Corresponding author. Y-H. Xia, xiadoc@outlook.com;yhxia@zjnu.cn. Address: College of Mathematics and Computer Science,  Zhejiang Normal University, 321004, Jinhua, China}$
\\
{\small \textit{$^a$ College of Mathematics and Computer Science,  Zhejiang Normal University, 321004, Jinhua, China}}\\
{\small \textit{$^b$ Departamento de Matem\'aticas, Universidad de Chile, Santiago, Chile }}\\
{\small Email: luwj@zjnu.edu.cn;  pintoj.uchile@gmail.com;  yhxia@zjnu.cn.}
}

\date{\today}
\maketitle

\begin{abstract}
   {
   Hartman-Grobman theorem states that there is a  homeomorphism $H$ sending the solutions of the nonlinear system onto those of its linearization under suitable assumptions.
   %Palmer's linearization theorem \cite{Palmer} states that there is a  homeomorphism $H$ sending the solutions of the nonlinear system onto those of its linearization under suitable assumptions. After his work,
    Many mathematicians have made contributions to prove H\"older continuity of the homeomorphisms. However, is it possible to improve the H\"older continuity to Lipschitzian continuity?
    This paper gives a positive answer. We formulate the first result that {\bf the homeomorphism is Lipschitzian, but not $C^1$, while its inverse is merely H\"{o}lder continuous, but not Lipschitzian.} It is interesting that the regularity of the homeomorphism is different from its inverse. Moreover, some illustrative examples are presented to show the effectiveness of our results. Further, motivated by our example, we also propose a conjecture, saying, the regularity of the homeomorphisms is sharp and it could not be improved any more.
     }
 \\
{\bf Keywords:} {Hartman-Grobman theorem;   stable manifolds; linearization; Exponential dichotomies}\\
 {\bf  MSC2020:} 34C41; 34D09; 34D10
  %\subjclass[2000]{34K23 \and 34D30 \and 37C60 \and 37C55 \and 39A12}
\end{abstract}

\section{Introduction and motivation}

\subsection{Brief history of Hartman-Grobman theorem}

%Linear equations are mathematically well-understood but nonlinear systems are relatively difficult to investigate. For this reason, linearization of differential equations is very important.

 A pioneering work on the linearization traces back to Poincar\'e \cite{Poincare}.
 He proved the analytical conjugation  between an analytic diffeomorphism and its linear part near a hyperbolic fixed
point in the complex case. Siegel \cite{Siegel}, %Brjuno \cite{B}
 and Yoccoz \cite{Yoccoz} studied the case that eigenvalues of the linear part lie on the unit circle.
A basic contribution to the linearization probelm in the real case for autonomous differential equations is the Hartman-Grobman theorem (see \cite{Hartman} and \cite{Grobman}).
Palis \cite{Palis}, Pugh \cite{Pugh1}, Bates and Lu \cite{Lu2}, Lu \cite{Lu1}, Hein and Pr\"{u}ss \cite{Hein-Pruss1}, and { Zgliczy\'nski \cite{Zgl-HG}}
made contributions to the linearization problem on the infinite dimensional space. In particular,
 Bates and Lu \cite{Lu2} obtained a Hartman-Grobman theorem for Cahn-Hilliard equation and phase field equations.  Lu \cite{Lu1}  proved a Hartman-Grobman theorem for the scalar reaction-diffusion equations.
 Hein and Pr\"{u}ss \cite{Hein-Pruss1} gave a version of Hartman-Grobman theorem for semilinear hyperbolic evolution equation on Banach space.
Palmer  \cite{Palmer1} firstly  extended the Hartman-Grobman theorem  to the nonautonomous case.
In order to weaken  Palmer's linearization theorem, various versions of Hartman-Grobman theorem were established, Backes et al. \cite{BDK-JDE} (for nonhyperbolic systems), Barreira and Valls \cite{B-V1,B-V2,B-V3,B-V4} (with nonuniform exponential dichotomies),
Huerta et al. \cite{Huerta1,Huerta2} (nonuniform exponential contraction),  Jiang \cite{Jiang1} (generalized exponential dichotomy),
Jiang  \cite{Jiang2} (ordinary dichotomy), Fenner and Pinto \cite{Fenner-Pinto} and Xia et al. \cite{Xia1} (for impulsive systems), Papaschinopoulos \cite{Papa-A} (for differential equations with piecewise constant argument), P\"otzche \cite{Potzche1} (for dynamic equations on time scales),  Reinfelds and Sermone \cite{Reinfelds1}, Reinfelds and  \v{S}teinberga \cite{Reinfelds-IJPAM} (dynamical equivalence), Shi and Zhang \cite{Shi-Zhang1} (monograph for linearization), Xia et al. \cite{Xia-BSM} (with unbounded nonlinear term).
Except for the $C^0$ linearization mentioned above, much effort was made to investigate $C^{r}$
linearization for $C^{k}$ ($1\leq r\leq k\leq \infty$) diffeomorphisms. Sternberg \cite{Sternberg1,Sternberg2} initially studied the smooth linearization problem.
    Recently, the smooth linearization for $C^{k}$  ($1\leq k\leq \infty$) diffeomorphisms are well improved by
     Sell \cite{Sell1},
Belitskill et al. \cite{Belitskii1,Belitskii3}, Cuong et al. \cite{Cuong}, Dragi\v{c}evi\'{c} et al. \cite{ZWN-MZ,DZZ-PLMS},  Elbialy \cite{ElBialy1}, Rodrigues and Sol\`{a}-Morales \cite{R-S-1,R-S-2,R-S-3},
Zhang et al. \cite{ZWN-JDE,ZWN-ETDS,ZWN-MA,ZWN-TAMS}. In particular, a set of nice results on the sharp regularity of linearization for  hyperbolic diffeomorphisms were established in Zhang et al. \cite{ZWN-JDE,ZWN-ETDS,ZWN-MA}.

\subsection{Motivations and novelty}

An important and interesting problem is the regularity of the linearization, which have greatly attracted many mathematicians'  attentions.
Among the works on the linearization mentioned above, a lot of papers were devoted to proving the H\"older continuity of the homeomorphisms in the linearization theorem (see Backes et al. \cite{BDK-JDE},
Barreira and Valls \cite{B-V1,B-V2,B-V3,B-V4}, Dragi\v{c}evi\'{c} et al. \cite{ZWN-MZ,DZZ-PLMS},
Huerta et al. \cite{Huerta1,Huerta2}, Hein and Pr\"{u}ss \cite{Hein-Pruss1}, Jiang \cite{Jiang1,Jiang2}, P\"otzche \cite{Potzche1}, Shi and Zhang \cite{Shi-Zhang1}, Rodrigues and Sol\`{a}-Morales \cite{R-S-1,R-S-2}, Xia et al. \cite{Xia1,Xia-BSM},
Zhang et al. \cite{ZWN-JDE,ZWN-ETDS,ZWN-MA}, Tan \cite{Tan-JDE}, Shi and Xiong \cite{Shi}). For the sake of easier illustration, we restate the Palmer's linearization theorem \cite{Palmer1} which has extended the classical Hartman-Grobman theorem (\cite{Hartman,Grobman}) to the nonautonomous case. It states that there is a  homeomorphism $H$ sending the solutions of the nonlinear perturbed system
\begin{equation}
y'=A(t)y+f(t,y). \label{nueva_1}
\end{equation}
 onto those of its linearization
 \begin{equation}
 x'(t)=A(t)x(t) \label{eq1}
\end{equation}
 under suitable assumptions.
Many mathematicians have made contributions to prove that both of the homeomorphisms are H\"older continuous.
However, is it possible to improve the H\"older continuity to Lipschitzian continuity? Up till now, there is no existing results on the  Lipschitzian continuity of the homeomorphisms in the Hartman-Grobman theorem. This paper gave a positive answer.  In this paper, we formulate the first result  that {\bf the homeomorphism $H$ is Lipschitzian, but not $C^1$, while its inverse $G=H^{-1}$  is merely H\"{o}lder continuous, but not Lipschitzian.} Moreover, some illustrative examples are presented to show the effectiveness of our results. Further, motivated by our example, we also propose a conjecture, saying, the regularity the homeomorphisms is sharp and it can not be improved any more.

Maybe, one would doubt that the regularity of the homeomorphism $H$ is different from its inverse. A simple example gives the answer. If $H(x)=x^2,(x>0)$ (locally Lipschitzian), then the inverse is $G(y)=y^{1/2}$ (H\"older continuous).

It is not standard to prove the Lipschitzian continuity of homeomorphism $H$. To overcome the difficulty, we have to use the dichotomy inequality as well as  the theory of stable manifolds and  unstable manifolds.

%Moreover, we weaken some conditions of the linearization theorem.
In the global version of Hartman-Grobamn (type) theorem (for example, Palmer's linearization theorem), it usually requires that the nonlinear term $f$ is uniformly bounded and Lipschitzian.
  In this paper, we also weaken the linearization theorem in two ways:
(i) we consider nonlinear terms $f$ which may be unbounded or not Lipschitzian (see Example \ref{example_2.1});
(ii) we prove that it is enough to assume the boundedness of the Green operator of the coefficients.

\subsection{Mechanism of improving the regularity}

    Standardly, to prove the regularity of the homeomorphisms, one takes direct estimates of the constructing homeomorphisms (e.g. \cite{B-V1,B-V2,B-V3,B-V4}) or employs the Bellman inequality (see e.g.
    \cite{Hein-Pruss1,BDK-JDE,Huerta2,Xia1,Xia-BSM,Potzche1,Shi}).
     However, the disadvantage of the Bellman inequality   results in an exponential estimate of the form $e^{\alpha t} (\alpha>0)$.
    It is expansive, which leads us to obtain H\"{o}lder regularity.
    Therefore, most of the previous works on the regularity of homeomorphisms of Hartman-Grobaman theorem in $C^0$ linearization  is H\"{o}lder continuous.

    On the contrary, the advantage of the dichotomy inequality \cite{Pinto1,Pinto2,Pinto3,Pinto4}  results in an exponential decay of the form
$e^{-\alpha_{1} t} (\alpha_{1}>0)$, see Lemma \ref{lemma_3.9} in the present paper.
    Thus, by dichotomy inequality, we can prove the Lipschitz continuity of the conjugacy.

\subsection{Organization of the paper}

    The rest of this paper is organized as follows:
    In Section 2, we present our main results, i.e.   regularity of the linearization   and illustrative examples. Also a conjecture on the sharpness is given.
    In Section 3, rigorous proofs are given to show our main results.

\section{Main results, illustrative examples and open conjecture}

\subsection{Notations and concepts}
Consider the following two nonautonomous systems
\begin{equation}
 x'=f(t,x) \label{(2.1)}
\end{equation}
and
\begin{equation}
y'=g(t,y)\label{(2.2)}
\end{equation}
where $x,y\in\mathbb{R}^{n}, t\in\mathbb{R}$.

\begin{definition}\label{definition_2.1}
Suppose that there exists a function $H:\mathbb{R}\times \mathbb{R}^{n}\to \mathbb{R}^{n}$ such that\\
(i) for each fixed $t$, $H(t,\cdot)$ is a homeomorphism of $\mathbb{R}^{n}$ into $\mathbb{R}^{n}$;\\
(ii) $\|H(t,x)-x\|$ is uniformly bounded with respect to $t$;\\
(iii) $G(t,\cdot)=H^{-1}(t,\cdot)$ also has property (ii);\\
(iv) if $x(t)$ is a solution of the system \eqref{(2.1)}, then $H(t,x(t))$ is a solution of the system \eqref{(2.2)}; and if $y(t)$ is a solution of the system \eqref{(2.2)}, then $G(t,y(t))$ is a solution of the system \eqref{(2.1)}.\\
If such a map $H_t:= H(t,x(t))$ exists, then the system \eqref{(2.1)} is topologically conjugated to the system \eqref{(2.2)} and the transformation $H(t,x)$ is called an equivalent function.
\end{definition}

\begin{definition}\label{definition_2.2}(Coppel \cite{[4]})
The linear system $x'=A(t)x$ is said to possess an exponential dichotomy, if there exist a projection $P(s)$ and constants $K>0,\alpha>0$  such that
\begin{eqnarray}
 \|U(t,s)P(s)\|&\leq& K \exp\{-\alpha(t-s)\},\qquad t\geq s,\label{(2.3)}\\
 \|U(t,s)(I-P(s))\|&\leq& K \exp\{\alpha(t-s)\},\qquad t\leq s, \nonumber
 \end{eqnarray}
hold; here $U(t,s):=U(t)U^{-1}(s)$ and $U(t)$ is a fundamental matrix of linear system $x'=A(t)x$.
\end{definition}
Let the Green function
$$k(t,s) = \left\{
     \begin{array}{lr}
       U(t,s)P(s),\qquad \quad  \text{if } &  t\geq s\\
       -U(t,s)(I-P(s)),\,\,\, \text{if } & t\leq s,
      \end{array}
   \right.$$
and the Green operator
\begin{equation*}
 \mathcal{K}(\phi)(t)=\int_{-\infty}^{\infty}k(t,s)\phi(s)ds,\quad t\in \mathbb{R}, \label{green}
\end{equation*}
where $\phi:\mathbb{R}\to \mathbb{R}^{n}$ is a function, $\|\mathcal{K}(\phi)\|\leq \mathcal{L}_{\alpha}(\|\phi\|)$ with
\begin{equation}\label{b-eq}
  \mathcal{L}_{\alpha}(b)(t)=\int_{-\infty}^{\infty}\exp\{-\alpha |t-s|\}b(s)ds,
\end{equation}
for $b:\mathbb{R}\to (0,\infty)$ a continuous function.

\subsection{Dichotomy inequality }

The following lemma will be useful in the rest of the present work. It consist of a dichotomic inequality developed by Pinto
\cite{Pinto1,Pinto2,Pinto3,Pinto4}. They are of the following type
\begin{equation}
 u(t)\leq c\exp\{-\alpha (t-t_0)\}+c_1 \int_{t_0}^t \exp\{-\alpha (t-\tau)\}b(\tau)u(\tau)d\tau +c_2 \int_{t}^s \exp\{-\alpha (\tau-t)\}b(\tau)u(\tau)d\tau, \label{dic1}
\end{equation}
\begin{equation}
 u(t)\leq c\exp\{-\alpha (s-t)\}+c_1 \int_{t_0}^t \exp\{-\alpha (t-\tau)\}b(\tau)u(\tau)d\tau +c_2 \int_{t}^s \exp\{-\alpha (\tau-t)\}b(\tau)u(\tau)d\tau, \label{dic2}
\end{equation}
where $\alpha,\,c,\,c_1$ and $c_2$ are positive constants.\\
For $\alpha_1>0$, define for $t\in(t_0,s)$
\begin{equation}
 \mathcal{L}_{\alpha_1}(b)(t)=c_1 \int_{t_0}^t \exp\{-\alpha_1 (t-\tau)\}b(\tau)d\tau +c_2 \int_{t}^s \exp\{-\alpha_1 (\tau-t)\}b(\tau)d\tau. \label{operador_l}
\end{equation}
Assume that
\begin{equation}
\sup_{t\in(t_0,s)} \mathcal{L}_{\alpha_1}(b)(t)=\theta_1<1. \label{condicion_theta}
\end{equation}
%%AQUI LA DESIGUALDAD DICOTOMICA
\begin{lemma}\label{lemma_3.1.1}[First dichotomic inequality]
Let $t_0\in \mathbb{R},s\in [t_0,\infty)$ and $u:[t_0,s)\to[0,\infty)$ be continuous, bounded for $s=\infty$ functions such that for $t\in[t_0,s)$ inequality \eqref{dic1} holds.
Then, $\forall \alpha_2<\alpha=\alpha_1+\alpha_2$ and $\forall t\in[t_0,s)$ we have
\begin{equation*}
 u(t)\leq \dfrac{c}{1-\theta_1}\exp\{-\alpha_2(t-t_0)\}.
\end{equation*}
\end{lemma}

\begin{lemma}\label{lemma_3.1.2}[Second dichotomic inequality]
Let $s\in \mathbb{R},t_0\in (-\infty,s]$ and $u:[t_0,s)\to[0,\infty)$ be continuous, bounded for $t_0=-\infty$ functions such that for $t\in(t_0,s]$ inequality \eqref{dic2} holds.
Then, $\forall \alpha_2<\alpha=\alpha_1+\alpha_2$ and $\forall t\in(t_0,s]$ we obtain
\begin{equation*}
 u(t)\leq \dfrac{c}{1-\theta_1}\exp\{-\alpha_2(s-t)\}.
\end{equation*}
\end{lemma}
\begin{proof}
$\forall \alpha_2<\alpha=\alpha_1+\alpha_2$, we use $\exp\{-\alpha\}=\exp\{-\alpha_2\}\cdot\exp\{-\alpha_1\}$. Since $\alpha_2(s-t)-\alpha_2(t-\tau)-\alpha_2(s-\tau)=2\alpha_2(t-\tau)\leq 0,$ for $t\geq \tau$
and $\alpha_2(s-t)-\alpha_2(\tau-t)-\alpha_2(s-\tau)=0,$ inequality \eqref{dic2} implies that $\hat{u}(t)=:u(t)\exp\{\alpha_2(s-t)\}$ satisfies \eqref{dic2} with $\alpha_1$ instead of $\alpha$:
\begin{equation}
 \hat{u}(t)\leq c\exp\{-\alpha_1 (s-t)\}+c_1 \int_{t_0}^t \exp\{-\alpha_1 (t-\tau)\}b(\tau)\hat{u}(\tau)d\tau +c_2 \int_{t}^s \exp\{-\alpha_1 (\tau-t)\}b(\tau)\hat{u}(\tau)d\tau. \nonumber
\end{equation}
Then
\begin{equation*}
 \hat{u}(t)\leq c\exp\{-\alpha_1(s-t)\}+ \mathcal{L}_{\alpha_1}(b)(t)\cdot \sup_{\tau\in[t_0,s]}\hat{u}(\tau)
\end{equation*}
and hence
\begin{equation*}
 \sup_{\tau\in[t_0,s]}\hat{u}(\tau)\leq \dfrac{c}{1-\theta_1}.
\end{equation*}
Therefore
\begin{equation*}
 u(t)\leq \dfrac{c}{1-\theta_1}\exp\{-\alpha_2(s-t)\},\quad \text{for}\;t\in [t_0,s].
\end{equation*}
Lemma 3 follows in a similar way.
\end{proof}
% \begin{corollary}
%   \begin{equation*}
%   u(t)\leq \dfrac{c}{1-\theta}\exp\{-\alpha_2|t-\tau|\},\qquad \text{for }t\geq \tau,\quad t\leq \tau.
%  \end{equation*}
% \end{corollary}

\subsection{Main results on the Hartman-Grobman theorem and its regularity}

We divide our statements of the main results into two parts. One is on the existence of homeomorphisms, the other is on the higher regularity of homeomorphism.

\subsubsection{Existence of homeomorphisms}
Consider the system \eqref{nueva_1} where $y\in\mathbb{R}^{n},$ $A(t)$ is a $n\times n$ continuous matrix defined on $\mathbb{R}$ and $f(t,y)$ a continuous function on $\mathbb{R}\times \mathbb{R}^{n}$ respectively.
The following global linearization theorem is for the existence of the homeomorphisms.

\begin{theorem}\label{theorem_2.1} ({\bf global linearization})
Suppose that \eqref{eq1} admits an exponential dichotomy of the form \eqref{(2.3)} on $\mathbb{R}$, and there exist nonnegative integrable functions $\mu(t),r(t)$ such that for all $t,x,\bar{x}$, $f(t,x)$ satisfies
\begin{eqnarray}\label{(2.5)}
 &&\|f(t,x)-f(t,\bar{x})\|\leq r(t)\|x-\bar{x}\|,\\
 &&\|f(t,x)\|\leq \mu(t),\nonumber
\end{eqnarray}
where $\mu(t), r(t)$ satisfy
\begin{equation}\label{c1_c2}
\sup_{t\in\mathbb{R}}\mathcal{L}_{\alpha}(\mu)(t)<\infty \quad \mathrm{and} \quad \sup_{t\in\mathbb{R}}\mathcal{L}_{\alpha}(r)(t)=\theta<K^{-1}.
\end{equation}
Then system \eqref{nueva_1} is topologically conjugated to its linear system
\begin{equation}
x'=A(t)x\label{(2.6)}
\end{equation}
and the equivalent function $H(t,x)$ and its inverse $G(t,x)$  satisfy
\begin{equation}
 \|H(t,x)-x\|\leq K\| \mathcal{L}_{\alpha}(\mu)\|_\infty, \label{estimacion_H}
\end{equation}
\begin{equation}
 \|G(t,x)-x\|\leq K\| \mathcal{L}_{\alpha}(\mu)\|_\infty, \label{estimacion_G}
\end{equation}
where
\begin{equation*}
 \|\mathcal{L}_{\alpha}(\mu)\|_{\infty}=\sup_{t\in\mathbb{R}}\mathcal{L}_{\alpha}(\mu)(t).
\end{equation*}
\end{theorem}

{Now we introduce a local version of Hartman-Grobman theorem.
The following lemma is elementary.

\begin{lemma}\label{local}
For some $\epsilon>0$, if $F:\mathbb{R}\times \overline{B}_\epsilon(0) \rightarrow \mathbb{R}^{n}$ satisfies $F(t,0)=0$ and
 \[ \|F(t,x)-F(t,\bar{x})\|\leq r(t)\|x-\bar{x}\|,\]
where $\overline{B}_\epsilon(0)$ is a small closed ball around zero and $r(t)$ is nonnegative local integrable with $\sup_{t\in\mathbb{R}}\mathcal{L}_{\alpha}(r)(t)=\theta<K^{-1}$.
Then the radial extension $\tilde{F}(t,x)$ defined by
\[
 \tilde{F}(t,x)=\left\{
           \begin{array}{ll}
             F(t,x), & x\in\overline{B}_\epsilon(0), \\
             F(t,\epsilon\frac{x}{\|x\|}), & x\in \mathbb{R}^{n}\backslash \overline{B}_\epsilon(0),
           \end{array}
         \right.
\]
always satisfies
\begin{equation}\label{ff}
  \|\tilde{F}(t,x)-\tilde{F}(t,\bar{x})\|\leq 2r(t)\|x-\bar{x}\|.
\end{equation}
\end{lemma}

\begin{theorem}\label{theorem-local}({\bf local linearization})
Suppose that \eqref{eq1} has an exponential dichotomy of the form \eqref{(2.3)} on $\mathbb{R}$.
Furthermore, for some $\epsilon>0$, if $f:\mathbb{R}\times \overline{B}_\epsilon(0) \rightarrow \mathbb{R}^{n}$ satisfies $f(t,0)=0$,
\[ \|f(t,x)-f(t,\bar{x})\|\leq 2r(t)\|x-\bar{x}\|,\]
and such that $2\sup_{t\in\mathbb{R}}\mathcal{L}_{\alpha} (r)(t)=2\theta<K^{-1}$.
Then system \eqref{nueva_1} is topologically conjugated to system \eqref{(2.6)} on $\overline{B}_\epsilon(0)$.
\end{theorem}

%\begin{corollary}\label{theorem-local2}
%Suppose that \eqref{eq1} has an exponential dichotomy of the form \eqref{(2.3)} on $\mathbb{R}$, and
 %there exists nonnegative local integrable function $r(t)$ such that for all $t,x,\bar{x}$, $\tilde{f}(t,x)$ satisfies $\tilde{f}(t,0)\equiv 0$ and
%\[ \|\tilde{f}(t,x)-\tilde{f}(t,\bar{x})\|\leq r(t)\|x-\bar{x}\|,\]
%where $r(t)$ satisfies $\sup_{t\in\mathbb{R}}\mathcal{L}(r)(t)=\theta_1<K^{-1}$ and $\tilde{f}(t,x)$  is defined by
%\begin{equation*}
 % \tilde{f}(t,x)=\left\{
  %         \begin{array}{ll}
   %          r(t)x, & \|x\|\leq \epsilon, \\
    %         0, & \|x\|\geq \delta,
     %      \end{array}
      %   \right.
%\end{equation*}
%for some arbitrarily chosen $0<\epsilon<\delta$. We assume that $\tilde{f}$ connects these two value smoothly.\\
%Then system \eqref{nueva_1} is topologically conjugated to system \eqref{(2.6)} at $\|x\|\leq \epsilon$.
%\end{corollary}

\begin{remark}
Clearly, Theorem \ref{theorem-local} is a local version of  Hartman-Grobman Theorem. Note that when $x\in\overline{B}_\epsilon(0)$,
\[ \|f(t,x)\|\leq r(t)\|x\|\leq r(t) \epsilon:=\mu(t).\]
when $x\in \mathbb{R}^{n}\backslash \overline{B}_\epsilon(0)$,
\[ \|f(t,x)\|\leq 2r(t) \epsilon:=2\mu(t).\]
If $\mu(t)$ satisfies
$\sup_{t\in\mathbb{R}}\mathcal{L}_{\alpha}(\mu)(t)<\infty$, then it satisfies all conditions of Theorem \ref{theorem_2.1}.
\end{remark}
}

\begin{lemma}\label{lemma_3.1}(Coppel \cite{[4]})
If $\mu(t)$, $r(t)$   are nonnegative local integrable functions on $\mathbb{R}$, i.e.,
$C_{\mu}=\sup_{t\in\mathbb{R}}\int_t^{t+1}\mu(s)ds$,  $C_r=\sup_{t\in\mathbb{R}}\int_t^{t+1}r(s)ds$, then we have
\begin{equation}\label{(3.1)}
\mathcal{L}_{\alpha}(\mu)(t)\leq 2 (1-\exp\{-\alpha\})^{-1}C_{\mu},\quad \mathrm{and} \quad
\mathcal{L}_{\alpha}(r)(t)\leq 2 (1-\exp\{-\alpha\})^{-1}C_{r}.
\end{equation}
\end{lemma}
Then the following result is obvious.
\begin{corollary}
 Theorem \ref{theorem_2.1} or Theorem \ref{theorem-local} hold if $C_{\mu}=\sup_{t\in\mathbb{R}}\int_t^{t+1}\mu(s)ds<\infty$ and
 $$\theta:=2(1-\exp\{-\alpha\})^{-1}C_r<K^{-1}.$$
\end{corollary}

\subsubsection{Higher regularity of homeomorphism}

Now it is the position to state our main result on the  regularity of homeomorphisms existing in Theorem \ref{theorem_2.1} and \ref{theorem-local}.
{
 \begin{theorem}\label{theorem_2.2} ({\bf regularity of homeomorphisms})
 Suppose that the conditions in Theorem \ref{theorem_2.1} or Theorem \ref{theorem-local} are satisfied.
If $\sup_{t\in \mathbb{R}} \mathcal{L}_{\alpha_1}(r)(t)=\tilde{\theta}<K^{-1}, $
   then the equivalent function $H$ is Lipschitzian, but its inverse $G$ is H\"{o}lder continuous.
   More specifically, there exist positive constants $p, q>0$ and $0<\beta<1$  such that
   % for any $x, \bar{x}\in \mathbb{R}^{n}$ (i.e., $P(t)x, P(t)\bar{x} \in P(t)(\mathbb{R}^{n})$ and
   % $(I-P(t))x, (I-P(t))\bar{x} \in (I-P(t))(\mathbb{R}^{n})$)
   \begin{equation}\label{(2.7)}
   \begin{cases}
   \|H(t,x)-H(t,\bar{x})\| \leq %\frac{p}{2} (\|P(t)(x-\bar{x})\|+\|(I-P(t))(x-\bar{x})\|) \leq
                               p\|x-\bar{x}\|, \\
   \|G(t,x)-G(t,\bar{x})\|\leq q\|x-\bar{x}\|^{\beta}.
   \end{cases}
   \end{equation}
  % holds for all $t, x$ and $\bar{x}$.
   %\item[(2)]
   %If $\sup\limits_{t\in\mathbb{R}}\int_{t}^{t+1}\|A(s)\|ds:=M$,
    % \end{description}
\end{theorem}
}
\begin{remark}
In the previous literature \cite{Hein-Pruss1,BDK-JDE,Huerta2,Jiang1,Xia1,Potzche1,Xia-BSM,Shi}, it is proven that both of the homeomorphisms are H\"older continuous. It can be restated that there exist
   positive constants $p_1, q>0$ and $0<\gamma, \beta<1$  such that
   \[\label{Holder-continuous}
   \begin{cases}
   \|H(t,x)-H(t,\bar{x})\|\leq p_1\|x-\bar{x}\|^{\gamma}, \\
   \|G(t,x)-G(t,\bar{x})\|\leq q \|x-\bar{x}\|^{\beta}.
  \end{cases}
\]
\end{remark}

{   When $A(t)\equiv A$, $A$ is a constant matrix, the systems reduce to the autonomous systems. Then we have the following corollary.} %下面添加自治系统的
{
\begin{corollary}\label{global-hg}% ({\bf global linearization})
 Let $A$ be hyperbolic, i.e., the spectrum of $A$ has no purely imaginary eigenvalues.
 If the nonlinear term $f$ satisfies
\[ \|f(x)-f(\bar{x})\|\leq r\|x-\bar{x}\|, \quad \|f(x)\|\leq \mu,\]
for all $x,\bar{x}\in\mathbb{R}^n$, and such that $2rk<\alpha$ ($k,\alpha$ are given in \eqref{(2.3)}), then the nonlinear autonomous
system $x'=Ax+f(x)$ is topologically conjugated to $x'=Ax$.\\
% and the equivalent function $H(x)$ and its inverse $G(x)$ satisfy
%\[ \|H(x)-x\|\leq 4k\mu\alpha^{-1}, \quad \|G(x)-x\|\leq 4k\mu\alpha^{-1}.\]
Moreover,the  homeomorphism $H(x)$ is Lipschitzian, but the inverse $G(x)$ is H\"{o}lder continuous, i.e.,  for $x,\bar{x}\in\mathbb{R}^n$, there exist positive constants $p,q>0,0<\beta<1$ such that
\begin{equation*}
   \begin{cases}
   \|H(x)-H(\bar{x})\| \leq  p\|x-\bar{x}\|, \\
   \|G(x)-G(\bar{x})\|\leq q\|x-\bar{x}\|^{\beta}.
   \end{cases}
   \end{equation*}
%where $p,q,\beta$ are given in Theorem \ref{theorem_2.2}.
\end{corollary}

%  Also, we have a local version of linearization  for Corollary \ref{global-hg}.
%\begin{corollary}\label{local-hg} ({\bf local linearization})
 %Let $A$ be hyperbolic.
 %If the nonlinear term $f$ satisfies $f(0)=0$ and $\|f(x)-f(\bar{x})\|\leq r\|x-\bar{x}\|$
%for all $x,\bar{x}\in\mathbb{R}^n$, and such that $4rk<\alpha$, then the nonlinear autonomous
% system $x'=Ax+f(x)$ is topologically conjugated to $x'=Ax$ on $\overline{B}_{\epsilon}(0)$.
%\end{corollary}

%\begin{corollary}\label{regularity-hg} ({\bf regularity of homeomorphisms})
  %Suppose that the conditions in Corollary \ref{global-hg} or Corollary \ref{local-hg} are satisfied.
  %If there exist $\xi_1=x-\bar{x} \in P(\mathbb{R}^{n})$ (or $\xi_2=x-\bar{x} \in (I-P)(\mathbb{R}^{n})$),
   %Then for $x,\bar{x}\in\mathbb{R}^n$
%\begin{equation*}
   %\begin{cases}
   %\|H(x)-H(\bar{x})\| \leq  p\|x-\bar{x}\|, \\
   %\|G(x)-G(\bar{x})\|\leq q\|x-\bar{x}\|^{\beta},
   %\end{cases}
   %\end{equation*}
%where $p,q,\beta$ are given in Theorem \ref{theorem_2.2}.
%\end{corollary}
}

\begin{corollary}
  In the nonuniform case, that is, the linear system admits a nonuniform exponential dichotomy instead of the uniform exponential dichotomy (\cite{B-V1,B-V2,B-V3,B-V4,Xia-Zhang}), Theorems \ref{theorem_2.1}, \ref{theorem-local} and \ref{theorem_2.2} are true for
$$\tilde{\mu}(t)=\mu(t)\exp\{-\epsilon|t|\} \quad \mathrm{and} \quad \tilde{r}(t)=r(t)\exp\{-\epsilon|t|\},$$
where $\mu(t)$ and $r(t)$ are given in \eqref{(2.5)}.
\end{corollary}
\noindent
In fact, we can see that
 \[ \mathcal{L}_{\alpha}(\tilde{\mu})(t)=\int_{-\infty}^{\infty}\exp\{-\alpha |t-s|+\epsilon|s|\}\mu(s)\exp\{-\epsilon|s|\}ds
    =\mathcal{L}_{\alpha}(\mu)(t),
 \]
similarly, $ \mathcal{L}_{\alpha}(\tilde{r})(t)=\mathcal{L}_{\alpha}(r)(t)$. Thus, all conditions of these theorems in the nonuniform case are satisfied.

% \begin{proof}
% We prove the first inequality (the other one is similar). For each natural number
% $m$, we obtain
% \begin{eqnarray}
%  &\displaystyle{\int_{t-(m+1)}^{t-m}\mu(s)\exp\{-\alpha(t-s)\}ds} & \leq \int_{t-(m+1)}^{t-m}\mu(s)\exp\{-\alpha m\}ds \nonumber \\
%  &&\leq C_1\exp\{-\alpha m\}.
%  \end{eqnarray}
% Therefore
% \begin{eqnarray}
%  &\displaystyle{\int_{-\infty}^{t}\mu(s)\exp\{-\alpha(t-s)\}ds}&=\sum_{m\in [0,\infty)}\int_{t-(m+1)}^{t-m}\mu(s)\exp\{-\alpha(t-s)\}ds\nonumber \\
%  &&\leq \sum_{m\in [0,\infty)}C_1 \exp\{-\alpha m\}\nonumber \\
%  &&=C_1 \{1-\exp(-\alpha)\}^{-1}\nonumber \\
%  &&=:V_1.
% \end{eqnarray}
% \end{proof}

\begin{remark}\label{remark_2.1}
Lemma \ref{lemma_3.1} shows that a big class of functions $\mu,r$ satisfy condition \eqref{c1_c2}.
$r,\mu\in L^p, 1\leq p\leq \infty, r$ with $\|r\|_p$ small enough.
If $r$ is uniformly bounded, it is possible to choose $\alpha=0$.
Note when $\mu(t)=\mu$ and $r(t)=r$ are constants, Theorem \ref{theorem_2.1} reduces to the classical Palmer linearization theorem.
We note that Palmer did not give a conclusion on the Lipschitz nor the H\"older continuity of $H$ .
It should be noted that $f(t,x)$ in our theorem could be unbounded or not uniformly Lipschitzian.
Note $\mu(t),r(t)$ are locally integrable (satisfying \eqref{c1_c2}), so they could be unbounded.
\end{remark}

\subsection{Illustrative examples and open conjecture}

\subsubsection{Illustrative examples to verify the higher regularity of homeomorphisms}

{
\newtheorem{Exam1}{Example}[section]
\begin{Exam1}\label{example2}
This example on the {\em global linearization}  shows that the homeomorphism $H$ is Lipschitzian, but its inverse is merely H\"older continuous.
\end{Exam1}
We consider the  hyperbolic equations on the unit circle, i.e.,
\begin{equation}\label{hyper-eq}
  \left(
    \begin{array}{c}
      x'_1 \\
      x'_2 \\
    \end{array}
  \right)=
  \left(
    \begin{array}{cc}
      -1 & 0 \\
      0 & 1 \\
    \end{array}
  \right)
  \left(
    \begin{array}{c}
      x_1 \\
      x_2 \\
    \end{array}
  \right)+
  \left(
    \begin{array}{c}
      f_1(x_1) \\
      f_2(x_2) \\
    \end{array}
  \right),
\end{equation}
where $f_1(x_1)$ and $f_2(x_2)$ is given by
\[
f_1(x_1)=
\begin{cases}
\epsilon x_1,&  0\leq x_1\leq 1,\\
 \epsilon x_1^{3},&  -1\leq x_1< 0,
\end{cases}
\quad \mathrm{and} \quad
f_2(x_2)=
\begin{cases}
-\epsilon x_2,&  0\leq x_2\leq 1,\\
 -\epsilon x_2^{3},&  -1\leq x_2< 0.
\end{cases}
\]
It is easy to see that $f$ is bounded and Lipschitzian.
Moreover, it is easy to obtain that  equation \eqref{hyper-eq} is topologically conjugated to its linear part
\begin{equation}\label{hyper-lin-eq}
  \left(
    \begin{array}{c}
      y'_1 \\
      y'_2 \\
    \end{array}
  \right)=
  \left(
    \begin{array}{cc}
      -1 & 0 \\
      0 & 1 \\
    \end{array}
  \right)
  \left(
    \begin{array}{c}
      y_1 \\
      y_2 \\
    \end{array}
  \right).
  \end{equation}
Thus the main purpose here is to construct an explicit formulae for $H=(H_1,H_2)^{T}$ and its inverse $G=(G_1,G_2)^{T}=H^{-1}$.
 For $t\geq 0$, we firstly consider the subsystem
\begin{equation}\label{n1}
  x'_1=-x_1+f_1(x_1).
\end{equation}
Notice that
\[
-x_1+f_1(x_1)
\begin{cases}
  <0, & \mathrm{if}\; 0<x_1\leq 1, \\
  =0, & \mathrm{if}\; x_1=0, \\
  >0, & \mathrm{if}\; -1\leq x_1< 0.
\end{cases}
\]
Thus a solution is either always $0$, always $<0$ or always $>0$.

  Firstly, $H_1(0)=0$ since $0$ is a solution of \eqref{n1} and $H_1(0)$ is the unique solution of $y'_1=-y_1$. But $0$ is such a solution.

  Secondly, we consider $0<x_1(t)\leq 1$. Clearly, $x_1(t)$ is strictly decreasing, i.e., $x_1(t)\rightarrow 0$ as $t\rightarrow +\infty$;
$x_1(t)\rightarrow 1$ as $t\rightarrow 0$.
Therefore, there must exists a unique time $t_{0}$ such that $x_1(t_{0})=1$. We set $t_{0}=0$.
If $t>0$, then $0<x_1(t)<1$ and so $x_1'(t)=-x_1(t)+\epsilon x_1(t)$ with $x_1(0)=1$. Hence,
\[ x_1(t)=e^{(-1+\epsilon)t}, \quad t\geq 0.\]
We need to find the unique solution $y_1 (t)$ of $y'_1=-y_1$ such that $\|y_1 (t)-x_1(t)\|$ is bounded.
Looking at $x_1(t)$ when $t\geq 0$, we see that $y_1 (t)=(1-\epsilon)e^{-t}$.
Hence for all $t\geq 0$,
\[H_1(x_1(t))=(1-\epsilon)e^{-t}.\]
Then
\[ H_1(1)=H_1(x_1(0))=1-\epsilon.\]
If $0<\xi_1<1$, then there exists a unique time $t>0$ such that $x_1(t)=e^{(-1+\epsilon)t}=\xi_1$. Then
\[ H_1(\xi)=H_1(x_1(t))=(1-\epsilon)e^{-t}=(1-\epsilon)x_1(t)^{\frac{1}{1-\epsilon}}
    =(1-\epsilon)\xi_1^{\frac{1}{1-\epsilon}}.\]
Therefore,
\[
H(x_1)=(1-\epsilon)x_1^{\frac{1}{1-\epsilon}},   \quad 0<x_1\leq 1.
\]

 Let us now consider $-1\leq x_1(t)<0$, clearly $x_1(t)$ is strictly increasing, i.e., $x_1(t)\rightarrow 0$ as $t\rightarrow \infty$;
$x_1(t)\rightarrow -1$ as $t\rightarrow0$.
So there must exists a unique time $t_{0}$ such that $x_1(t_{0})=-1$. We set $t_{0}=0$.
If $t>0$, then $-1<x_1(t)<0$ and so $x_1'(t)=-x_1(t)+\epsilon x_1^{3}(t)$ with $x_1(0)=-1$. Letting $z_1=x_1^{-2}$, then $z_1'=2z-2\epsilon$ with $z_1(0)=1$. Hence,
$z_1(t)=(1-\epsilon)e^{2t}+\epsilon$, that is,
\[ x_1(t)=-\left[(1-\epsilon)e^{2t}+\epsilon\right]^{-\frac{1}{2}}, \quad t\geq 0.\]
We need to find the unique solution $y_1(t)$ of $y'_1=-y_1$ such that $|y_1(t)-x_1(t)|$ is bounded.
Looking at $x_1(t)$ when $t\geq 0$, we see that $y_1(t)=(\epsilon-1)e^{-t}$.
Hence for all $t\geq 0$,
\[H_1(x_1(t))=(\epsilon-1)e^{-t}.\]
Then
\[ H_1(1)=H_1(x_1(0))=\epsilon-1.\]
If $-1<\xi_1<0$, then there exists a unique time $t>0$ such that
\[x_1(t)=-\left[(1-\epsilon)e^{2t}+\epsilon\right]^{-\frac{1}{2}}=\xi_1.\]
 Then
\[ H_1(\xi_1)=H_1(x_1(t))=(\epsilon-1)e^{-t}=-(1-\epsilon)^{\frac{3}{2}}\left((-x_1(t))^{-2}-\epsilon\right)^{-\frac{1}{2}}
    =-(1-\epsilon)^{\frac{3}{2}}\left((-\xi_1)^{-2}-\epsilon\right)^{-\frac{1}{2}}.\]
Thus  we obtain that
\[
H_1(x_1)= -(1-\epsilon)^{\frac{3}{2}}\left((-x_1)^{-2}-\epsilon\right)^{-\frac{1}{2}},  \quad -1\leq x_1< 0.
\]
Summarizing we have found that
\[
H_1(x_1)=
\begin{cases}
   (1-\epsilon)x_1^{\frac{1}{1-\epsilon}}, &  0<x_1\leq 1, \\
  0, & x_1=0,\\
 -(1-\epsilon)^{\frac{3}{2}}\left((-x_1)^{-2}-\epsilon\right)^{-\frac{1}{2}}, &  -1\leq x_1< 0.
\end{cases}
\]
%where $0<\epsilon<\frac{1}{3}$.
%Similar to the procedure just shown
We next claim that $H_1$ is a continuous function, but it is not $C^{1}$. In fact, we only say that $H_1$ is continuous at $0$,
but is not $C^{1}$ at $0$.\\
$H_1(x_1)$ is continuous at $x_1=0$:
   \[\begin{split}
    &\lim\limits_{x_1\rightarrow 0^{+}} (1-\epsilon)x_1^{\frac{1}{1-\epsilon}}=0,\\
    &\lim\limits_{x_1\rightarrow 0^{-}} -(1-\epsilon)^{\frac{3}{2}}\left((-x_1)^{-2}-\epsilon\right)^{-\frac{1}{2}}=0.
   \end{split}\]
Hence, $H_1(x_1)$ is continuous, but the following fact proves that $H_1$ is not $C^{1}$ at $0$. Clearly,
for $0<x_1\leq 1$, $H'_1(x_1)=x_1^{\frac{\epsilon}{1-\epsilon}}$,
and for $ -1\leq x_1< 0$, $H'_1(x_1)=(1-\epsilon)^{\frac{3}{2}}(1-\epsilon (-x_1)^{2})^{-\frac{3}{2}}$. Therefore,
   \[\begin{split}
    &\lim\limits_{x_1\rightarrow 0^{+}} x_1^{\frac{\epsilon}{1-\epsilon}}=0,\\
    &\lim\limits_{x_1\rightarrow 0^{-}}  (1-\epsilon)^{\frac{3}{2}}(1-\epsilon (-x_1)^{2})^{-\frac{3}{2}}=(1-\epsilon)^{\frac{3}{2}},
   \end{split}\]
which implies that $H_1(x_1)$ is not in $C^{1}$.
Fortunately, $H_1(x_1)$ is Lipschitz continuous,
since $H'_1(x_1)$ is continuous at $x_1$ except for $x_1=0$, and it is bounded with $\|H'_1\|\leq 1$.
So function $H_1(x_1)$ is globally Lipschitz continuous with Lipschitz constant $L = 1$, but is not in $C^{1}$.

However, the inverse function $G_1=H_1^{-1}$ is
\[
G_1(y_1)=
\begin{cases}
(\frac{y_1}{1-\epsilon})^{1-\epsilon}, &  0<y_1\leq 1-\epsilon, \\
  0, & y_1=0,\\
 -\left((1-\epsilon)^{-\frac{4}{3}}(-y_1)^{-2}+\epsilon\right)^{-\frac{1}{2}}, &  -1+\epsilon\leq y_1< 0.
\end{cases}
\]
Obviously, $G_1(y_1)$ is continuous at $0$. So $G_1(y_1)$ is a continuous function. However, this is not Lipschitz continuous since $y_1^{1-\epsilon}$ is not Lipschitz, as $0<1-\epsilon<1$.

Secondly, for $t\leq 0$, we consider the subsystem
\[ x_2'=x_2+f_2(x_2).\]
Similar to the procedure just shown, we can obtain that
\[
H_2(x_2)=
\begin{cases}
   (1-\epsilon)x_2^{\frac{1}{1-\epsilon}}, &  0<x_2\leq 1, \\
  0, & x_2=0,\\
 -(1-\epsilon)^{\frac{3}{2}}\left((-x_2)^{-2}-\epsilon\right)^{-\frac{1}{2}}, &  -1\leq x_2< 0,
\end{cases}
\]
and
\[
G_2(y_2)=
\begin{cases}
(\frac{y_2}{1-\epsilon})^{1-\epsilon}, &  0<y_2\leq 1-\epsilon, \\
  0, & y_2=0,\\
 -\left((1-\epsilon)^{-\frac{4}{3}}(-y_2)^{-2}+\epsilon\right)^{-\frac{1}{2}}, &  -1+\epsilon\leq y_2< 0.
\end{cases}
\]
Hence, $H_2$ is Lipschitzian, but $G_2$ is only H\"{o}lder continuous.
Therefore,  Theorem \ref{theorem_2.2} is verified.% in this particular case,.
%\end{Exam1}

\begin{Exam1}\label{example_2.3}
This example on the {\em local linearization} is to show that the homeomorphism $H$ is Lipschitzian, but its inverse is merely H\"older continuous.
\end{Exam1}
We  consider the following non-autonomous system
\begin{equation}\label{ex-1}
  x'=-x+f(t,x),
\end{equation}
where $f(t,x)$ is given by
\begin{equation*}
  f(t,x)=\left\{
           \begin{array}{ll}
             0, & \|x\|\leq \epsilon, \\
             \frac{2 e^{-t}}{e^{t}+e^{-t}}x, & \|x\|\geq \delta,
           \end{array}
         \right.
\end{equation*}
for some arbitrarily chosen $0<\epsilon<\delta$. We assume that $f$ connects these two value smoothly.
Thus the vector field of Eq. \eqref{ex-1} is nonlinear. When $\|x\|\leq \epsilon$, it is identical to the linear flow.
Hence, we only need to limit ourselves to $\|x\|\geq \delta$. For $\|x\|\geq \delta$,
we can check that  $x(t)=\frac{2}{e^{t}+e^{-t}}$ is a bounded solution with the initial value $x(0)=1$.

Now we set $H(t,x)=\frac{1}{x}-\frac{e^{t}}{2}$. Notice that
\[ H(t,x(t))=\frac{1}{x(t)}-\frac{e^{t}}{2}=\frac{e^{t}+e^{-t}}{2}-\frac{e^{t}}{2}=\frac{1}{2} e^{-t},\]
which implies that $H(t,x(t))$ is a solution of $y'=-y$.
To show its regularity, take any $\|x_1\|, \|x_2\| \geq \delta$, we have
\[\|H(t,x_1)-H(t,x_2)\|= \left\|\frac{1}{x_1}-\frac{1}{x_2}\right\| \leq \frac{1}{\delta^{2}} \|x_1-x_2\|. \]
It means that $H$ is Lipschitzian for $\|x\|\geq \delta$. Moreover for $\|y\|\leq \frac{1}{\delta}$, $G:=H^{-1}=\frac{2}{e^{t}+2y}$ and
\[\|G(t,y_1)-G(t,y_2)\|=\frac{4\|y_2-y_1\|}{(e^{t}+2y_1)(e^{t}+2y_2)},\]
when $t\rightarrow -\infty$, $\|G(t,y_1)-G(t,y_2)\|=\frac{\|y_2-y_1\|}{\|y_1 y_2\|} \rightarrow \infty$.
Therefore, $G$ is not Lipschitzian. If we take $\|y_1-y_2\|<1$, then there exists $0<q<1$ such that
 $\|G(t,y_1)-G(t,y_2)\|\leq \|y_1-y_2\|^{q}$.

\begin{Exam1}\label{example_2.1}
The following example shows that $f(t,x)$ in our conditions \eqref{(2.5)} could be unbounded, nor uniformly Lipschitzian. Thus, it is weaker than previous works on the Palmer's linearization theorem.
\end{Exam1}
We construct a continuous function $f(t,x)$ which is unbounded, not uniformly Lipschitzian, but locally integrable.
Considering $[0,\infty)$, for any positive constant $c$ and integer $m$, let
\[
\bar{g}(t)=
     \begin{cases}
       0, &\quad\text{if } t\in [0,1),\\
       cm^2t-cm^3, &\quad\text{if } t\in \left[m,m+\tfrac{1}{2m}\right), \\
       -cm^2t+cm^3+cm, &\quad\text{if } t\in \left[m+\tfrac{1}{2m},m+\tfrac{1}{m}\right), \\
       0, &\quad\text{if } t\in \left[m+\tfrac{1}{m},m+1\right).\\
     \end{cases}
\]
Note that $\bar{g}(t)$ is continuous on $[0,\infty)$. Let $\mu(t)$ the continuous function on $\mathbb{R}$:
\[
\mu(t)=
     \begin{cases}
       \bar{g}(t), &\quad\text{if } t\geq 0,\\
       \bar{g}(-t), &\quad\text{if } t<0.\\
     \end{cases}
\]
Thus,
$$f(t,x)=\mu(t)\sin(x)$$
is continuous on $\mathbb{R}\times \mathbb{R}^{2} $. It is easy to see that for any $(t,x),(t,\bar{x})\in \mathbb{R}\times \mathbb{R}^{2}$,
\begin{eqnarray*}
&\|f(t,x)-f(t,\bar{x})\|\leq \mu(t)\|x-\bar{x}\|,\\
&\|f(t,x)\|\leq \mu(t)
\end{eqnarray*}
and
 $$\int_t^{t+1}\mu(s)ds\leq c.$$
However, we see that $\mu$ and $f$ are unbounded functions, since $$\mu\left(m+\dfrac{1}{2m}\right)\to +\infty, \text{  as }m\to\infty. $$
Consequently, $f(t,x)$ is not only unbounded, but also $f(t,x)$ is not uniformly Lipschitzian.
%This example shows that  $f(t,x)$ could be unbounded or not uniformly Lipschitzian in condition \eqref{(2.5)} of our theorem.

\subsubsection{Open conjecture}

The above two illustrative example show that our main results on the higher regularity of homeomorphisms are correct.
That is, the homeomorphism $H$ is Lipschitzian, and the inverse of the homeomorphism is H\"older continuous. In particular, in Example \ref{example2}, it is shown that homeomorphism $H$ is Lipschitzian, {\bf but not $C^1$}; its inverse is H\"older continuous, {\bf but not Lipschitzian}. Moreover, it is difficult to verify that both the homeomorphism $H$ and its inverse in this example are unique, respectively. Therefore, from this example, we assert that   the homeomorphism $H$ is Lipschitzian and its inverse is H\"older continuous in the Hartman-Grobman, {\bf and the regularity of the homeomorphisms is sharp.  That is to say, the regularity of the homeomorphisms could not be improved any more.} But in this situation, it is only a conjecture from the example. We need a strict proof, but it is an open problem now.

}

\section{Preliminary results}
\subsection{Preliminary results for the existence of homeomorphisms}
In what follows, we always suppose that the conditions of Theorem \ref{theorem_2.1} are satisfied.
Let $X(t,t_0,x)$ be a solution of  system \eqref{nueva_1} satisfying the initial condition $X(t_0)=x$ and
$Y(t,t_0,y)$ is a solution of  system \eqref{(2.6)} satisfying the initial condition $Y(t_0)=y$ .
To prove the main results, we divide our proof into several lemmas.

\begin{lemma}\label{lemma_3.2}
For each $(\tau,\xi)$, the system
\begin{equation}
Z'=A(t)Z-f(t,X(t,\tau,\xi))\label{(3.2)}
\end{equation}
has a unique bounded solution $h(t,(\tau,\xi))$  with $\|h(t,(\tau,\xi))\|\leq K\parallel \mathcal{L}_{\alpha}(\mu)\parallel _\infty$.
\end{lemma}

\begin{proof}
For any fixed $(\tau,\xi))$, let
\begin{eqnarray*}
 &h(t,(\tau,\xi))=-\mathcal{K}(f(\cdot,X(\cdot,\tau,\xi))(t)=&-\int_{-\infty}^t U(t,s)P(s)f(s,X(s,\tau,\xi))ds\\
 &&+\int_{t}^{\infty}U(t,s)(I-P(s))f(s,X(s,\tau,\xi))ds.
\end{eqnarray*}
Differentiating it, it is easy to see that $h(t,(\tau,\xi))$ is a solution of the system \eqref{(3.2)}. It follows from \eqref{(2.3)} and \eqref{(2.5)} that
\begin{eqnarray*}
 &\|h(t,(\tau,\xi))\|&\leq \int_{-\infty}^t K \mu(s)\exp\{-\alpha (t-s)\}ds + \int_{t}^{\infty} K \mu(s)\exp\{\alpha (t-s)\}ds \\
 &&\leq K\parallel \mathcal{L}_{\alpha}(\mu)\parallel _\infty,
\end{eqnarray*}
which implies that $h(t,(\tau,\xi))$ is a bounded solution of the system \eqref{(3.2)}. We claim that the bounded solution is unique. In fact, for any fixed $(\tau,\xi)$ , the system \eqref{(3.2)} is linearly inhomeogeneous,
and its linear system $Z'=A(t)Z$ has an exponential dichotomy. This implies that the bounded solution of \eqref{(3.2)} is unique.
\end{proof}

\begin{lemma}\label{lemma_3.3}
For each $(\tau,\xi)$, the system
\begin{equation}
 Z'=A(t)Z+f(t,Y(t,\tau,\xi)+Z)\label{(3.3)}
\end{equation}
has a unique bounded solution $g(t,(\tau,\xi))$, and $\|g(t,(\tau,\xi))\|\leq K\parallel \mathcal{L}_{\alpha}(\mu)\parallel _\infty$.
\end{lemma}

\begin{proof}
Let $\mathbf{B}$ be the complete metric space of all the continuous bounded functions $Z(t)$, provided of supremum metric, with $\|Z(t)\|\leq K\parallel \mathcal{L}_{\alpha}(\mu)\parallel _\infty$.
For each $(\tau,\xi)$ and any $Z(t)\in \mathbf{B}$, define a mapping $\mathcal{T}$ as follows,
\begin{eqnarray*}
 &\mathcal{T}Z(t)=&\int_{-\infty}^t U(t,s)P(s)f(s,Y(s,\tau,\xi)+Z(s))ds\\
 &&-\int_{t}^{\infty}U(t,s)(I-P(s))f(s,Y(s,\tau,\xi)+Z(s))ds.
\end{eqnarray*}
\noindent
A simple computation leads to
\begin{equation*}
 \parallel \mathcal{T}Z(t)\parallel \leq K\mathcal{L}_{\alpha}(\mu)(t),
\end{equation*}
which implies that $\mathcal{T}\mathbf{B}\subset \mathbf{B}$. For any $Z_1(t),Z_2(t)\in \mathbf{B}$,
\begin{equation*}
 \|\mathcal{T}Z_1 (t)-\mathcal{T}Z_2 (t)\|\leq K \mathcal{L}_{\alpha}(r)(t)\parallel Z_1-Z_2\parallel .
\end{equation*}
Now, by \eqref{c1_c2} $K\theta<1$, then $\mathcal{T}$ has a unique fixed point, namely $Z_0(t)$, and
\begin{eqnarray*}
 &Z_0(t)=&\int_{-\infty}^t U(t,s)P(s)f(s,Y(s,\tau,\xi)+Z_0(s))ds\\
 &&-\int_{t}^{\infty}U(t,s)(I-P(s))f(s,Y(s,\tau,\xi)+Z_0(s))ds.
\end{eqnarray*}
It is easy to see that $Z_0(t)$ is a bounded solution of the system \eqref{(3.3)}.
From standard argument, the bounded solution is unique. We may call the unique solution $g(t,(\tau,\xi))$.
From the above proof, it is easy to see that $\|g(t,(\tau,\xi))\|\leq K\parallel \mathcal{L}_{\alpha}(\mu)\parallel _\infty$.
\end{proof}
Similarly, we have:
\begin{lemma}\label{lemma_3.4}
Let $x(t)$ be any solution of the system \eqref{eq1}. Then $Z(t)\equiv 0$ is the unique bounded solution of the system
\begin{equation}
Z'=A(t)Z+f(t,x(t)+Z)-f(t,x(t)). \label{(3.4)}
\end{equation}
\end{lemma}
Note the importance in these results of the uniform boundedness of $\mathcal{L}_{\alpha}(\mu)(t).$\\
{\bf Constructing the homeomorphisms:}
Now we define two functions as follows
\begin{equation}
H(t,x)=x+h(t,(t,x)),\label{(3.5)}
\end{equation}
\begin{equation}
G(t,y)=y+g(t,(t,y)),\label{(3.6)}
\end{equation}
for $g$ and $h$ as in Lemmas \ref{lemma_3.2} and \ref{lemma_3.3}.

By differentiation and similar arguments, we have the following lemmas.

\begin{lemma}\label{lemma_3.5}
For any fixed $(t_0,x)$ , $H(t,X(t,t_0,x))$ is a solution of system \eqref{(2.6)}.
\end{lemma}

\begin{lemma}\label{lemma_3.6}
For any fixed $(t_0,y)$ , $G(t,Y(t,t_0,y))$ is a solution of  system \eqref{(2.7)}.
\end{lemma}

\begin{lemma}\label{lemma_3.7}
For any $t\in\mathbb{R},y\in\mathbb{R}^{n}, H(t,G(t,y))=y$.
\end{lemma}
\begin{proof}
Let $y(t)$ be any solution of  linear system \eqref{(2.6)}. From Lemma \ref{lemma_3.6}, $G(t,y(t))$ is a solution of  system \eqref{nueva_1}.
Then by Lemma \ref{lemma_3.5}, we see that $H(t,G(t,y(t)))$ is a solution of  system \eqref{(2.6)}, written as $\overline{y}(t)$. Let
\begin{equation*}
 J(t)=\overline{y}(t)-y(t).
\end{equation*}
To prove this conclusion, we need to show that $J(t)\equiv 0$. In fact, differentiating $J$, we have
\[\begin{split}
 J'(t)=&\overline{y}'(t)-y'(t)\\
 =&A(t)\overline{y}(t)-A(t)y(t)\\
 =&A(t)J(t),
\end{split}\]
which implies that $J$ is a solution of the system $Z'=A(t)Z$. From Lemma \ref{lemma_3.2} and Lemma \ref{lemma_3.3}, it follows that
\[\begin{split}
 \|J(t)\|=&\|\overline{y}(t)-y(t)\|\\
 =&\|H(t,G(t,(t,y(t)))-y(t)\|\\
 \leq&\|H(t,G(t,(t,y(t)))-G(t,(t,y(t)))\|+\|G(t,(t,y(t)))-y(t)\|\\
 \leq& 2K\parallel \mathcal{L}_{\alpha}(\mu)\parallel _\infty.
\end{split}\]
This implies that $J(t)$ is a bounded solution of the system $Z'=A(t)Z$. However, the linear system $Z'=A(t)Z$ has no nontrivial bounded solution.
Hence $J(t)\equiv 0$, that is,
$\overline{y}(t)=y(t)$.\\
Thus, $J(t)\equiv 0$, that is,
\begin{equation*}
 \overline{y}(t)=y(t),\quad \text{or}\quad H(t,G(t,y(t)))\equiv y(t).
\end{equation*}
Since $y(t)$  is an arbitrary solution of linear system \eqref{(2.6)}, the proof of Lemma \ref{lemma_3.7} is complete.
\end{proof}

\begin{lemma}\label{lemma_3.8}
For any $t\in \mathbb{R}, x\in\mathbb{R}^{n}$, we have
$$G(t,H(t,x))=x.$$
\end{lemma}
\begin{proof}
The proof is similar to that in Lemma \ref{lemma_3.7}.
\end{proof}

{
%\begin{lemma}\label{Copple-lem} [Coppel \cite{[4-2]}] (Schauder-Tychonoff theorem)
 %  Let $\mathcal{C}(\mathbb{J})$ be the set of all functions which are continuous on the interval $\mathbb{J}$,
% and let $\mathcal{F}$ be the subset of by those functions $x(t)$ such that $\|x(t)\|\leq \vartheta (t)$ for all $t\in \mathbb{J}$,
%where $\vartheta (t)$ is a fixed positive continuous function.
%\\
 % Let $\mathcal{J}$ be a mapping of $\mathcal{F}$ into itself with the properties:

%(i) $\mathcal{J}$ is continuous, in the sense that if $x_n \in \mathcal{F} (n=1,2,\cdots)$ and $x_n \rightarrow x$ uniformly on every
%compact subinterval of $\mathbb{J}$, then $\mathcal{J}x_n \rightarrow \mathcal{J} x$ uniformly on every compact subinterval of $\mathbb{J}$;

%(ii) the functions in the image set $\mathcal{J}\mathcal{F}$ are equicontinuous and bounded at every point of $\mathbb{J}$.
%\\
% Then the mapping $\mathcal{J}$ has at least one fixed point in $\mathcal{F}$.
%\end{lemma}

\subsection{Key lemma for the Lipchitzian continuity of homeomorphism}
{   To introduce our key lemma, we begin with a result from \cite{[4],Meiss-book,ZWN-Chinese}. We restate it as follows.
\begin{lemma}\label{moti-lemma}
  Assume that system \eqref{eq1} admits an exponential dichotomy with the form \eqref{(2.3)} on $\mathbb{R}$.\\
(1) If system  \eqref{nueva_1} has a bounded  solution $X(t,t_0,x)$ on $[t_0,\infty)$ satisfying the initial value $X(t_0)=x$, then
  $X(t,t_0,x)$  can be expressed by:
\begin{equation}\label{PPP}
 X(t,t_0,x)=  U(t,t_0)P(t_0)x+\int_{t_0}^t \Phi_{P}(t,\tau)f(\tau,X(\tau,t_0,x))d\tau +\int_t^{\infty}\Phi_{Q}(t,\tau)f(\tau,X(\tau,t_0,x))d\tau,
\end{equation}
where
$$\Phi_{P}(t,\tau)=U(t,\tau)P(\tau),\quad t\geq \tau, \quad \text{ and } \quad \Phi_{Q}(t,\tau)=-U(t,\tau)(I-P(\tau)),\quad t\leq \tau. $$
%Moreover,  $X(t,t_0,x)$ is bounded  on $[t_0,\infty)$.
Conversely, all solutions $X(t,t_0,x)$ of \eqref{PPP} on $[t_0,\infty)$ are the solutions of \eqref{nueva_1}.\\
(2) If system  \eqref{nueva_1} has a   bounded solution $X(t,t_0,x)$ on $(-\infty, t_0]$ satisfying the initial value $X(t_0)=x$,  then
$X(t,t_0,x)$ can be expressed by:
\begin{equation}\label{QQQ}
X(t,t_0,x)=  U(t,t_0)(I-P(t_0))x+\int_{-\infty}^t\Phi_{P}(t,\tau)f(\tau,X(\tau,t_0,x))d\tau +\int_t^{t_0}\Phi_{Q}(t,\tau)f(\tau,X(\tau,t_0,x))d\tau.
\end{equation}
%and   $X(t,t_0,x)$ is bounded  on $[t_0,\infty)$.
Conversely, all solutions $X(t,t_0,x)$ of \eqref{QQQ}  on $(-\infty, t_0]$ are the solutions of \eqref{nueva_1}.
\end{lemma}

{ The following lemma plays a great role in the proof of  Lipchitzian continuity of homeomorphism.
 % The first part is the converse of Lemma \ref{moti-lemma},  and in the second part we obtain an estimate of exponential decay by means of the dichotomy inequality.
\begin{lemma}\label{lemma_3.9}
Denote $X(t,t_0,x)$ is the solution of system \eqref{nueva_1} satisfying  $X(t_0)=x\in \mathbb{R}^{n}$.\\
(i) For any $\xi_1 \in P(t_0)(\mathbb{R}^{n})$, % such that $\|\xi_1\|\leq (1-K\theta_1)K^{-1}M_1$, then
\eqref{nueva_1} has a unique bounded solution $X(t,t_0,x)$ on $[t_0,\infty)$ satisfying $P(t_0)X(t_0)=\xi_1$, which is expressed by \eqref{PPP}; %$\|X(t,t_0,x)\| \leq M_1$
(ii) For any $\xi_2 \in (I-P(t_0))(\mathbb{R}^{n})$,  \eqref{nueva_1} has a unique bounded solution $X(t,t_0,x)$ on $(-\infty, t_0]$ satisfying $(I-P(t_0))X(t_0)=\xi_2$, which is expressed by \eqref{QQQ}.\\
Moreover, if $\alpha=\alpha_1+\alpha_2$ and
$\sup_{t\in \mathbb{R}} \mathcal{L}_{\alpha_1}(r)(t)=\tilde{\theta}<K^{-1}, $
 then for any $\alpha_2<\alpha$ the following conclusions hold:\\
(1) For $P(t_0)(x-\bar{x}) \in P(t_0)(\mathbb{R}^{n})$, we have
\begin{equation}\label{21}
\quad \|X(t,t_0,x)-X(t,t_0,\overline{x})\|\leq \dfrac{K}{1-K\tilde{\theta}} \|P(t_0)(x-\overline{x})\|\exp\{-\alpha_2(t-t_0)\}, \quad t\geq t_0;
\end{equation}
(2) For $(I-P(t_0))(x-\bar{x})\in (I-P(t_0))(\mathbb{R}^{n})$, we have
\begin{equation}\label{21-2}
\|X(t,t_0,x)-X(t,t_0,\overline{x})\|\leq \dfrac{K}{1-K\tilde{\theta}} \|(I-P(t_0))(x-\overline{x})\|\exp\{\alpha_2(t-t_0)\}, \quad t\leq t_0.
\end{equation}
\end{lemma}
}
\begin{proof}
  We claim the first part by means of Banach contraction mapping principle.
  Let $\mathcal{BC}$ be the set of all bounded continuous functions defined for $t\geq t_0$.
%and for any $\varrho <M_1$ let $\mathbf{B}$ be the subset of those continuous functions $X(t)$ for which $\|X(t)\|\leq \varrho$ for $t\geq t_0$.
  If $\mathcal{J}$ is the mapping defined by
\[
\mathcal{J}X(t,t_0,x)=U  (t,t_0)P(t_0)x+\int_{t_0}^t\Phi_{P}(t,\tau)f(\tau,X(\tau,t_0,x))d\tau + \int_{t}^{\infty}\Phi_{Q}(t,\tau)f(\tau,X(\tau,t_0,x))d\tau.
\]
%where
%$$\Phi_{P}(t,\tau)=U(t,\tau)P(\tau),\quad t\geq \tau, \quad \text{ and } \quad \Phi_{Q}(t,\tau)=-U(t,\tau)(I-P(\tau)),\quad t\leq \tau. $$
By using \eqref{(2.5)} and \eqref{c1_c2}, it is easy to see that $\mathcal{J}X$ is continuous and bounded. In fact,
\[\begin{split}
  \|\mathcal{J}X(t,t_0,x)\| \leq& K \exp\{-\alpha(t-t_0)\} \|\xi_1\|+\int_{t_0}^{\infty} K \exp\{-\alpha|t-\tau|\} \mu(\tau) d\tau \\
        \leq & K\|\xi_1\|+K\sup_{\tau\geq t_0}\mathcal{L}_{\alpha}(\mu)(\tau)<\infty,
\end{split}\]
where $\xi_1=P(t_0)X(t_0)=P(t_0)x$.
%If we choose $\varrho< M_1$ so that $K\|\xi_1\|\leq (1-K\theta_1)\varrho$, then
%$\sup\limits_{t\geq t_0}\|X(t,t_0,x)\|\leq \varrho$ implies $\|\mathcal{J}X(t,t_0,x)\| \leq \varrho$.
Hence $\mathcal{J}$ maps $\mathcal{BC}$ into itself.
 Note that $K\theta<1$ (see \eqref{c1_c2}), for any $X_1 (t,t_0,x), X_2 (t,t_0,x)\in \mathcal{BC}$, we have that
(also using \eqref{(2.5)} and \eqref{c1_c2})
\[\begin{split}
 \|\mathcal{J}X_1 (t,t_0,x)-\mathcal{J}X_2 (t,t_0,x)\|\leq& \int_{t_0}^{\infty} K\exp\{-\alpha|t-\tau|\} r(\tau)
    \| X_1 (\tau,t_0,x)-X_2 (\tau,t_0,x)\|d\tau \\
     \leq& K\theta \sup\limits_{\tau\geq t_0} \| X_1 (\tau,t_0,x)-X_2 (\tau,t_0,x)\|,
\end{split}\]
which implies that $\mathcal{J}$ is a contraction mapping in $\mathcal{BC}$, that is, there is a unique fixed point $X^*=\mathcal{J}X^*$
such that $X^*$ is bounded for $t\geq t_0$ and $P(t_0)X(t_0)=\xi_1$. The expression follows from Lemma \ref{moti-lemma} immediately.

Similar to the above procedure for $t\leq t_0$, by replacing the mapping  by
\[
\mathcal{J}X(t,t_0,x)=U (t,t_0)\xi_2+\int_t^{t_0}\Phi_{Q}(t,\tau)f(\tau,X(\tau,t_0,x))d\tau + \int_{-\infty}^{t}\Phi_{P}(t,\tau)f(\tau,X(\tau,t_0,x))d\tau,
\]
where $\xi_2 =(I-P(t_0))x\in (I-P(t_0))(\mathbb{R}^{n})$.  We can show that $X(t)$ is solution of \eqref{nueva_1} with the required property (ii).

 We now show the second part by means of Lemma \ref{lemma_3.1.1} and Lemma \ref{lemma_3.1.2}.
%By the variation formula, we get for $t\geq t_0$:
%\begin{equation*}
% X(t,t_0,\xi)=\Phi_P (t,t_0)P(t_0)\xi+\int_{t_0}^t\Phi_{P}(t,\tau)f(\tau,X(\tau,t_0,\xi))d\tau + \int_{t}^{\infty}\Phi_{Q}(t,\tau)f(\tau,X(\tau,t_0,\xi))d\tau,
%\end{equation*}
%where $$\Phi_{P}(t,\tau)=U(t,\tau)P(\tau),\quad t\geq \tau \qquad \text{ and } \qquad \Phi_{Q}(t,\tau)=-U(t,\tau)(I-P(\tau)),\quad t\leq \tau. $$
By \eqref{(2.5)},   we conclude that for any initial condition on $P(t_0)(x-\bar{x}) \in P(t_0)(\mathbb{R}^{n})$
\begin{equation}\label{22}
\begin{split}
 \|X(t,t_0,x)-X(t,t_0,\overline{x})\|\leq & K\|P(t_0)(x-\overline{x})\|\exp\{-\alpha(t-t_0)\} \\
 &+\int_{t_0}^{\infty}K\exp\{-\alpha|t-\tau|\}r(\tau)\|X(\tau,t_0,x)-X(\tau,t_0,\overline{x})\|d\tau.
\end{split}
 \end{equation}
Since $K\tilde{\theta}<1$,  and using dichotomic inequality in Lemma \ref{lemma_3.1.1}, we obtain that for $t\geq t_0$
\begin{eqnarray*}
 &\|X(t,t_0,x)-X(t,t_0,\overline{x})\|\leq & \dfrac{K}{1-K\tilde{\theta}}\exp\{-\alpha_2(t-t_0)\}\|P(t_0)(x-\overline{x})\|.
\end{eqnarray*}
%for $t\geq t_0$, where $$\theta_1=\sup_{t\in\mathbb{R}}\int_{-\infty}^\infty \exp\{-\alpha_1 |t-\tau|\}r(\tau)d\tau<K^{-1}.$$
Finally, analogous analysis for $t\leq t_0$, dichotomic inequality in Lemma \ref{lemma_3.1.2} ends the proof.
\end{proof}
\subsection{An intuitive example to understand Lemma \ref{lemma_3.9}.}
\begin{Exam1}\label{example_2.0}
In this example, we give an intuitive example to understand Lemma \ref{lemma_3.9}. It is a key lemma to prove that the homeomorphism $H$ is Lipschitz continuous.
\end{Exam1}

For simplicity, we consider the following planar system
\begin{equation}\label{planar-eq}
  x'=Ax+f(x), \quad x\in \mathbb{R}^{2},
\end{equation}
where $f:\mathbb{R}^{2}\rightarrow \mathbb{R}^{2}$ and $A=\mathrm{diag} \{-1,1\}$.
If we take $P=\mathrm{diag} \{1,0\}$ and $I-P=\mathrm{diag} \{0,1\}$, then %\eqref{planar-eq} can be rewritten as
\begin{equation}\label{planar-eq1}
 Px'= \left(
    \begin{array}{c}
      x'_1 \\
      0 \\
    \end{array}
  \right)=\left(
            \begin{array}{cc}
              -1 & 0 \\
              0 & 0\\
            \end{array}
          \right)
  \left(
    \begin{array}{c}
      x_1 \\
      0 \\
    \end{array}
  \right)+
  \left(
    \begin{array}{c}
      f_1(x_1,x_2) \\
      0 \\
    \end{array}
  \right)
\end{equation}
and
\begin{equation}\label{planar-eq2}
 (I-P)x'= \left(
    \begin{array}{c}
     0 \\
       x'_2 \\
    \end{array}
  \right)=\left(
            \begin{array}{cc}
              0 & 0 \\
              0 & 1\\
            \end{array}
          \right)
  \left(
    \begin{array}{c}
      0 \\
       x_2 \\
    \end{array}
  \right)+
  \left(
    \begin{array}{c}
     0  \\
       f_2(x_1,x_2) \\
    \end{array}
  \right).
\end{equation}
Suppose that $f=(f_1,f_2)^T$  %is bounded and Lipschitzian,
satisfies $f(0)=0$ and $\|f(x)\|\leq \mu$, $\|f(x)-f(\bar{x})\| \leq r\|x-\bar{x}\|$ for all $x, \bar{x}\in\mathbb{R}^{2}$,
where $\mu>0$ and $r<\frac{1}{2}$.
Then we have $K=1, \alpha=1, \theta=2rK/\alpha=2r<1$. Thus for $t\geq 0$,
it is clear that equation  \eqref{planar-eq1} has a unique bounded solution $X(t)$ with the initial condition $(x_1(0),0)^{T}$
(i.e., $P(x_1 (0),x_2 (0))^{T}=(x_1(0),0)^{T}\in P(\mathbb{R}^{2})$). Note
\[\begin{split}
X(t)
% \left(
%   \begin{array}{c}
 %    x_1(t) \\
  %   0 \\
  % \end{array}
 %\right)
 =&
 e^{-t}\left(
         \begin{array}{cc}
           1 & 0 \\
           0 & 0\\
         \end{array}
       \right)\left(
    \begin{array}{c}
      x_1(0) \\
      0 \\
    \end{array}
  \right)
 +\int_0^{t} e^{-(t-s)}\left(
         \begin{array}{cc}
           1 & 0 \\
           0 & 0\\
         \end{array}
       \right)
       \left(
         \begin{array}{c}
           f_1(x_1(s),x_2(s)) \\
           0 \\
         \end{array}
       \right)ds  \\
  &- \int_{t}^{\infty} e^{t-s} \left(
         \begin{array}{cc}
           0 & 0 \\
           0 & 1\\
         \end{array}
       \right)
       \left(
         \begin{array}{c}
           0\\
            f_2(x_1(s),x_2(s)) \\
         \end{array}
       \right)ds.
\end{split}\]
Then for $t\geq 0$ and  $x,\bar{x}$,
\[\|X(t,0,x)-X(t,0,\bar{x})\|\leq c_1 e^{-\epsilon t}\|Px-P\bar{x}\|,\]
for some constants $c_1>0$, $0<\epsilon<1$.
For $t\leq 0$, similar to the procedure just shown, we have that  $X(t)$ is bounded and
\[\|X(t,0,x)-X(t,0,\bar{x})\|\leq c_2 e^{\epsilon t}\|(I-P)x-(I-P)\bar{x}\|,\]
for some  constants $c_2>0$, $0<\epsilon<1$.
Therefore, Lemma \ref{lemma_3.9} always holds.

}
%We deduce the case $t\leq t_0$ from the first part $t\geq t_0$ inversing the order of the time. Without restriction, consider $t_0=0$. Let $t^*=-t$ and $x^*(t)=x(-t)=x(t^*)$
%$$x'(t)=A(t)x(t)+f(t,x(t))$$
%and $x(t)=U(t)v(t)$ satisfies
%$$v'(t)=U^{-1}(t)f(t,U(t),v(t)):=g(t)=g^*(+t^*).$$
%Denoting $\dot{()}=\tfrac{d}{dt^*}$, we have $v'(t)=-(\dot{v}^*)(t^*).$
%Hence, $$\dot{v}^*(t^*)=-g^*(t^*)$$
%which integrated on $[0,t^*]$ gives
%$$v^*(t^*)-v^*(0)=-\int_0^{t^*}g^*(s^*)ds^*$$
%or
%\begin{equation}
%x^*(t^*)=U(t^*)U^{-1}(0)x(0)-\int_0^{t^*}U(t^*)U^{-1}(s^*)f(s^*,x^*(s^*))ds^*. \label{24}
%\end{equation}
%Because $v^*(t^*)=v(-t^*)=U^{-1}(-t^*)x(-t^*)=U^{-1}(t^*)x^*(t^*).$\\
%In integral equation \eqref{24} we have the order corrected $0\leq s^*\leq t^*$ and the result follows from the first part of this proof
%\begin{eqnarray*}
% &&\left \|X^*(t^*,0,x_1)-X^*(t^*,0,x_2)\right\|\leq \left(\dfrac{K}{1-K\theta^*_1}\right)\|x_1-x_2\|\exp\{-\alpha_2 t^*\},\quad t^*\geq 0,\\
% &&\theta_1^*=\displaystyle{\sup_{t^*\in \mathbb{R}}\int_{-\infty}^\infty\exp\{-\alpha_2|t^*-s^*|\}r(s^*)ds^*<K^{-1}.}
%\end{eqnarray*}
%Then $\theta_1^*=\theta_1$ and
%\begin{eqnarray*}
% &&\left \|X(-t,0,x_1)-X(-t,0,x_2)\right\|\leq \left(\dfrac{K}{1-K\theta_1}\right)\|x_1-x_2\|\exp\{\alpha_2 t\},\quad t< 0.
%\end{eqnarray*}
}
\section{Proofs of main results}

\subsection{Proofs of Theorem \ref{theorem_2.1}}
Now we are in a position to prove  Theorem \ref{theorem_2.1}.
\begin{proof}[Proof of Theorem \ref{theorem_2.1}]
Now we show that $H(t,\cdot)$ satisfies the four conditions of Definition \ref{definition_2.1}.\\
For any fixed $t$, it follows from Lemma \ref{lemma_3.7}, \ref{lemma_3.8} that $H(t,\cdot)$ is homeomorphism and $G(t,\cdot)=H^{-1}(t,\cdot)$. Thus, Condition $(i)$ is satisfied.
From \eqref{(2.5)} and Lemma \ref{lemma_3.2}, we derive $\|H(t,x)-x\|=\|h(t,(t,x))\|\leq K\parallel \mathcal{L}_{\alpha}(\mu)\parallel _\infty$.
Note $\|H(t,x)\|\to\infty$  as $\|x\|\to\infty$, uniformly with respect to $t$. Thus, Condition $(ii)$ is satisfied.
From \eqref{(2.6)} and Lemma \ref{lemma_3.3}, we derive $\|G(t,y)-y\|=\|g(t,(t,y))\|$.
Note $\|G(t,y)\|\to\infty$  as $\|y\|\to\infty$, uniformly with respect to $t$. Thus, Condition $(iii)$ is satisfied.
From Lemma \ref{lemma_3.5}, Lemma \ref{lemma_3.6}, we know that Condition $(iv)$ is true.
Hence, the system \eqref{nueva_1} and its linear system \eqref{(2.6)} are topologically conjugated.
This completes the proof of  Theorem \ref{theorem_2.1}.
\end{proof}

\subsection{Proofs of Theorem \ref{theorem_2.2}}
Now we are in a position to prove Theorem \ref{theorem_2.2}. We divide the proof into two steps.
\begin{proof}
\textbf{Proof (1).} \textbf{Step 1-1}
    %Note that in Theorem \ref{theorem_2.2}, {\color{red} %%%%红色可以删掉？？~~
%we assume that the Cauchy problem of system \eqref{nueva_1} is bounded,
%which is a necessary condition to satisfy the dichotomy inequality in Lemma \ref{lemma_3.9}. }
   % Thus,
We are going to use dichotomy inequality to prove the Lipshcitz continuity of the equivalent function $H$.
We claim that
$$\|H(t,x)-H(t,\overline{x})\|\leq p\|x-\overline{x}\|, \qquad p=1+\frac{2K^2\tilde{\theta}}{1-K\tilde{\theta}}.$$
By uniqueness $X(t,(\tau,\xi))=X(t,(t,X(t,(\tau,\xi)))$.
From Lemma \ref{lemma_3.2}, it follows that
\begin{eqnarray*}
 &h(t,(t,\xi))=&-\int_{-\infty}^t U(t,s)P(s)f(s,X(s,t,\xi))ds\\
 &&+\int_{t}^{\infty}U(t,s)(I-P(s))f(s,X(s,t,\xi))ds,
\end{eqnarray*}
which is also equivalent to
\[\begin{split}
 P(t)h(t,(t,\xi))=& -\int_{-\infty}^t U(t,s)P(s)f(s,X(s,t,\xi))ds, \\
 (I-P(t))h(t,(t,\xi))=& \int_{t}^{\infty}U(t,s)(I-P(s))f(s,X(s,t,\xi))ds.
\end{split}\]
Thus we get
\[\begin{split}
I_1:=& P(t)h(t,(t,\xi))-P(t)h(t,(t,\overline{\xi})) \\
=& \int_{-\infty}^t U(t,s)P(s)(f(s,X(s,t,\overline{\xi}))-f(s,X(s,t,\xi)))ds,\\
I_2:=& (I-P(t))h(t,(t,\xi))-(I-P(t))h(t,(t,\overline{\xi})) \\
=& \int_{t}^{\infty}U(t,s)(I-P(s))(f(s,X(s,t,\xi))-f(s,X(s,t,\overline{\xi})))ds.
\end{split}\]
%\begin{eqnarray*}
 %&h(t,(t,\xi))-h(t,(t,\overline{\xi}))=&-\int_{-\infty}^t U(t,s)P(s)(f(s,X(s,t,\xi))-f(s,X(s,t,\overline{\xi})))ds \qquad \textbf{$(I_1)$}\\
 %&&+\int_{t}^{\infty}U(t,s)(I-P(s))(f(s,X(s,t,\xi))-f(s,X(s,t,\overline{\xi})))ds\quad \textbf{$(I_2)$}.
%\end{eqnarray*}
In view of Lemma \ref{lemma_3.9}, %the Cauchy problem problems
for any initial condition on $P(t_0)(\mathbb{R}^n)$ (or $(I-P(t_0))(\mathbb{R}^n)$) of system \eqref{nueva_1} is bounded on demiaxes $[s,\infty)$ (or $(-\infty,s]$). %of system \eqref{nueva_1} is bounded.
Then, by  condition \eqref{(2.5)}, and using Lemma \ref{lemma_3.9} (part $t\leq t_0$), we deduce that
\begin{eqnarray*}
 &\|I_1\|\leq & \int_{-\infty}^t K\exp\{-\alpha(t-s)\}\dfrac{K}{1-K\tilde{\theta}}\exp\{\alpha_2(t-s)\}\parallel (I-P(t))(\xi-\overline{\xi})\parallel r(s)ds\\
 &&\leq \dfrac{K^2}{1-K\tilde{\theta}}\parallel (I-P(t))(\xi-\overline{\xi})\parallel \int_{-\infty}^t \exp\{-\alpha_1(t-s)\}r(s)ds,
%  &&\leq \dfrac{K^2 \tilde{\theta}}{1-K\tilde{\theta}}\parallel \xi-\overline{\xi}\parallel .
\end{eqnarray*}
and similarly, by Lemma \ref{lemma_3.9} (part $t\geq t_0$), we have that
\begin{eqnarray*}
 &\|I_2\|\leq & \int_{t}^{\infty} K\exp\{\alpha(t-s)\}\dfrac{K}{1-K\tilde{\theta}}\exp\{-\alpha_2(t-s)\}\parallel P(t)( \xi-\overline{\xi})\parallel r(s)ds\\
 &&\leq \dfrac{K^2}{1-K\tilde{\theta}}\parallel P(t)( \xi-\overline{\xi})\parallel \int_{t}^{\infty} \exp\{\alpha_1(t-s)\}r(s)ds.
%  &&\leq \dfrac{K^2 \tilde{\theta}}{1-K\tilde{\theta}}\parallel \xi-\overline{\xi}\parallel .
\end{eqnarray*}
Hence, we conclude that
\[\begin{split}
 \|I_1\|+\|I_2\|\leq & \dfrac{K^2}{1-K\tilde{\theta}}(\Vert P(t)( \xi-\overline{\xi})\Vert+\Vert (I-P(t))( \xi-\overline{\xi})\Vert) \\
 &\left(\int_{-\infty}^t \exp\{-\alpha_1(t-s)\}r(s)ds+ \int_{t}^{\infty} \exp\{\alpha_1(t-s)\}r(s)ds\right).
%  &&\leq \dfrac{K^2 \tilde{\theta}}{1-K\tilde{\theta}}\parallel \xi-\overline{\xi}\parallel .
\end{split}\]
From \eqref{c1_c2} and the above inequality, it follows that
\[\begin{split}
&\|(P(t)+I-P(t)) (h(t,(t,\xi))-h(t,(t,\overline{\xi})))\| \\
\leq& \|I_1\|+\|I_2\|
\leq \left(\dfrac{K^2 \tilde{\theta}}{1-K\tilde{\theta}}\right)
(\Vert P(t)( \xi-\overline{\xi})\Vert+\Vert (I-P(t))( \xi-\overline{\xi})\Vert).
\end{split}\]
%
%$$\parallel (P(t)+I-P(t)) (h(t,(t,\xi))-h(t,(t,\overline{\xi})))\parallel \leq \left(\dfrac{K^2 \tilde{\theta}}{1-K\tilde{\theta}}\right)
%(\Vert P(t)( \xi-\overline{\xi})\Vert+\Vert (I-P(t))( \xi-\overline{\xi})\Vert).$$
By the definition of $H(t,x)$,
\[
\begin{split}
 \|H(t,x)-H(t,\overline{x})\|\leq &\|x-\overline{x}\|+\dfrac{K^2 \tilde{\theta}}{1-K\tilde{\theta}}(\|P(t)(x-\overline{x})\|+\|(I-P(t))(x-\overline{x})\|)\\
 \leq& \left(1+\dfrac{2K^2 \tilde{\theta}}{1-K\tilde{\theta}}\right)\|x-\overline{x}\|\\
 \equiv& p\|x-\overline{x}\|.
\end{split}
\]
This completes the proof of Step 1-1.\\
 \textbf{Step 1-2} We show that there exist positive constants $q>0$ and $0<\beta<1$ such that for all $t, y, \bar{y}$
 \[ \|G(t,y)-G(t,\bar{y})\| \leq q\|y-\bar{y}\|^{\beta}.\]
 Usually this point is treated with successive approximations.
 From Lemma \ref{lemma_3.3}, we know that $g(t,(\tau,\xi))$ is a fixed point of the following map $\mathcal{T}$
\begin{equation}\label{g-expression}
\begin{split}
 (\mathcal{T}Z)(t)=&\int_{-\infty}^t U(t,s)P(s)f(s,Y(s,\tau,\xi)+Z(s))ds\\
 &-\int_{t}^{\infty}U(t,s)(I-P(s))f(s,Y(s,\tau,\xi)+Z(s))ds.
 \end{split}
\end{equation}
Let $g_{0}(t,(\tau,\xi))\equiv 0$, and by recursion define
\[\begin{split}
 g_{m+1}(t,(\tau,\xi))=& \int_{-\infty}^t U(t,s)P(s)f(s,Y(s,\tau,\xi)+g_{m}(s,(\tau,\xi)))ds\\
     &-\int_{t}^{\infty}U(t,s)(I-P(s))f(s,Y(s,\tau,\xi)+g_{m}(s,(\tau,\xi)))ds
\end{split} \]
  It is not difficult to show that
\[ g_{m}(t,(\tau,\xi))\rightarrow  g(t,(\tau,\xi)), \quad \text{as} \ \ m\rightarrow +\infty, \]
  uniformly with respect to $t, \tau,\eta$.
\\
  Note that $g_{0}(t,(\tau,\xi))=g_{0}(t,(t,Y(t,\tau,\xi)))$.
  Thus, by induction, it is clear that for all $m (m\in\mathbb{N})$,
 $g_{m}(t,(\tau,\xi))=g_{m}(t,(t,Y(t,\tau,\xi)))$.
 Choose $\lambda>0$ sufficiently large and $\beta>0$ sufficiently small such that
$$
\begin{array}{l}
\lambda>\frac{3}{1-\exp \{-\alpha\}}+\frac{3}{2(1-\exp \{\alpha-M\})} ,\\
\beta<\frac{\alpha}{M+C_{\mu}}, \\
0<\frac{2 K C_{r}}{1-\exp \{-(\alpha-M \beta)\}}<\frac{1}{3},
\end{array}
$$
where $\alpha, K$ are given by \eqref{(2.3)}, $C_\mu, C_r$ are positive constants defined in \eqref{(3.1)} and $M=\sup_{t\in\mathbb{R}}\|A(t)\|$.
  Now we first show that if $0<\|\xi-\bar{\xi}\|<1$ for all $m,$ we have
\begin{equation}\label{g-assume-inequality}
\|g_{m}(t,(t, \xi))-g_{m}(t,(t, \bar{\xi}))\|<\lambda\|\xi-\bar{\xi}\|^{\beta}.
\end{equation}
  Obviously, inequality (\ref{g-assume-inequality}) holds if $m=0$. Now making the inductive assumption that (\ref{g-assume-inequality}) holds.
  From (\ref{g-expression}), it follows that
\begin{equation*}
\begin{split}
&g_{m+1}(t,(t,\xi))-g_{m+1}(t,(t,\bar{\xi}))\\
=& \int_{-\infty}^t U(t,s)P(s)[f(s,Y(s,t,\xi)+g_{m}(s,(t,\xi)))
- f(s,Y(s,t,\bar{\xi})+g_{m}(s,(t,\bar{\xi})))]ds \\
&-\int_{t}^{\infty}U(t,s)(I-P(s))[f(s,Y(s,t,\xi)+g_{m}(s,(t,\xi)))
- f(s,Y(s,t,\bar{\xi})+g_{m}(s,(t,\bar{\xi})))]ds \\
\triangleq & J_{1}+J_{2}.
\end{split}
\end{equation*}
We divide $J_{1}, J_{2}$ into two parts:
\[\begin{split}
 J_{1}=& \int_{-\infty}^{t-\tau} + \int_{t-\tau}^{t}\triangleq J_{11}+J_{12},\\
 J_{2}=& \int_{t}^{t+\tau} + \int_{t+\tau}^{\infty}\triangleq J_{21}+J_{22},
\end{split}\]
where $\tau=\frac{1}{M+C_\mu} \ln \frac{1}{\|\xi-\bar{\xi}\|}$. By \eqref{(2.3)}, \eqref{(2.5)} and \eqref{(3.1)}, we have
$$
\begin{aligned}
\|J_{11}\| & \leq \int_{-\infty}^{t-\tau} K \exp \{-\alpha(t-s)\} 2 \mu(s) d s \\
& =\sum_{m\in [0,\infty)}\int_{t-\tau-m-1}^{t-\tau-m} 2K \mu(s)\exp\{-\alpha(t-s) \}ds \\
&\leq \sum_{m\in [0,\infty)} 2KC_\mu \exp\{-\alpha(\tau+m) \} \\
& \leq 2KC_\mu \exp\{-\alpha \tau\}(1-\exp\{-\alpha\})^{-1} \\
& \leq \frac{2 K C_\mu}{1-\exp \{-\alpha\}}\|\xi-\bar{\xi}\|^{\frac{\alpha}{M+C_\mu}}, \\
\end{aligned}
$$
and similarly,
$$
\begin{aligned}
\|J_{22}\| & \leq \int_{t+\tau}^{+\infty} K \exp \{-\alpha(t-s)\} 2 \mu(s) d s \\
& \leq \frac{2 K C_\mu}{1-\exp \{-\alpha\}}\|\xi-\bar{\xi}\|^{\frac{\alpha}{M+C_\mu}}.
\end{aligned}
$$
When $0<\|\xi-\bar{\xi}\|<1$, $s\in [t-\tau,t]$ and since $\|Y(t,t_{0},y)-Y(t,t_{0},\bar{y})\|\leq \|y-\bar{y}\|\exp\{M|t-t_{0}|\}$,
we can get
$$
\begin{aligned}
 \|Y(s,t,\xi)-Y(s,t,\bar{\xi})\| &\leq \|\xi-\bar{\xi}\|\cdot \exp\{M|t-s| \} \\
 &\leq \|\xi-\bar{\xi}\|\cdot \exp\{M\tau \} \\
 &\leq \|\xi-\bar{\xi}\|^{\frac{M}{M+C_\mu}} <1.
\end{aligned}
$$
Hence, it is easy to see that
$$
\begin{aligned}
 \|g_{m}(s,(t,\xi))-g_{m}(s,(t,\bar{\xi}))\| &=\|g_{m}(s,(s,Y(s,t,\xi)))-g_{m}(s,(s,Y(s,t,\bar{\xi})))\| \\
 &\leq \lambda \|\xi-\bar{\xi}\|^{\beta}\cdot \exp \{ M\beta|t-s|\},
\end{aligned}
$$
Therefore, we have that
$$
\begin{aligned}
\|J_{12}\| \leq & \int_{t-\tau}^{t} K \exp \{-\alpha(t-s)\} r(s)[\|\xi-\bar{\xi}\|\\
&\left.\cdot \exp \{M(t-s)\}+\lambda \|\xi-\bar{\xi}\|^{\beta} \cdot \exp \{M \beta(t-s)\}\right] d s \\
=& \int_{t-\tau}^{t} K \exp \{(M-\alpha)(t-s)\} r(s)\|\xi-\bar{\xi}\| d s \\
&+\int_{t-\tau}^{t} \lambda K\|\xi-\bar{\xi}\|^{\beta} \cdot r(s) \cdot \exp \{(M \beta-\alpha)(t-s)\} d s \\
=& \sum_{m \in[0,[\tau]]} K \int_{t-\tau+m}^{t-\tau+m+1} \exp \{(M-\alpha)(t-s)\} r(s)\|\xi-\bar{\xi}\| d s \\
&+\sum_{m \in[0,[\tau]]} K \lambda \int_{t-\tau+m}^{t-\tau+m+1} \exp \{(M \beta-\alpha)(t-s)\} r(s)\|\xi-\bar{\xi}\|^{\beta} d s \\
\leq & \sum_{m \in[0,[\tau]]} K C_r \exp \{(M-\alpha)(\tau-m)\}\|\xi-\bar{\xi}\| \\
&+\sum_{m \in[0,[\tau]]} C_r K \lambda \exp \{(M \beta-\alpha) \tau\}\|\xi-\bar{\xi}\|^{\beta} \exp \{(M \beta-\alpha)(-m-1)\} \\
\leq & C_r K \exp \{(M-\alpha) \tau\} \frac{1}{1-\exp \{-(M-\alpha)\}}\|\xi-\bar{\xi}\| \\
&+K C_r \lambda \exp \{(M \beta-\alpha) \tau\} \frac{\exp \{\alpha-M \beta\}(1-\exp\{(\alpha-M\beta)[\tau]\})}{1-\exp \{\alpha-M \beta\}}\|\xi-\bar{\xi}\|^{\beta} \\
\leq & C_r K \exp \left\{(M-\alpha) \frac{-1}{M+C_\mu} \ln \|\xi-\bar{\xi}\|\right\}
 \times \frac{1}{1-\exp \{-(M-\alpha)\}}\|\xi-\bar{\xi}\| \\
&+K C_r \lambda \exp \{(M \beta-\alpha) \tau\} \frac{\exp \{(\alpha-M \beta)\tau\}}{1-\exp \{-(\alpha-M \beta)\}}\|\xi-\bar{\xi}\|^{\beta} \\
=& C_r K\|\xi-\bar{\xi}\|^{\frac{\alpha+C_\mu}{M+C_\mu}} \cdot \frac{1}{1-\exp \{-(M-\alpha)\}}
+K C_r \lambda \frac{1}{1-\exp \{-(\alpha-M \beta)\}}\|\xi-\bar{\xi}\|^{\beta},
\end{aligned}
$$
Note that $M-\alpha>0,-\alpha+M \beta<0$ imply that $\exp \{(M \beta-\alpha) \tau\}<1$ and $\beta<\frac{\alpha+C_\mu}{M+C_\mu}$.
Then
$$
\begin{aligned}
\|J_{12}\| \leq & K C_{r}\|\xi-\bar{\xi}\|^{\beta} \cdot \frac{1}{1-\exp \{-(M-\alpha)\}}
+K C_{r} \lambda \frac{1}{1-\exp \{-(\alpha-M \beta)\}}\|\xi-\bar{\xi}\|^{\beta} \\
=& K C_{r}\left[\frac{1}{1-\exp \{\alpha-M\}}+\frac{\lambda}{1-\exp \{-(\alpha-M \beta)\}}\right]\|\xi-\bar{\xi}\|^{\beta}.
\end{aligned}
$$
Similar arguments lead to
$$
\|J_{21}\| \leq K C_{r}\left[\frac{1}{1-\exp \{\alpha-M\}}+\frac{\lambda}{1-\exp \{-(\alpha-M \beta)\}}\right]\|\xi-\bar{\xi}\|^{\beta}.
$$
Hence,
\[
\begin{split}
&\|g_{m+1}(t,(t, \xi))-g_{m+1}(t,(t, \bar{\xi}))\| \\
\leq&\left[\frac{4 K C_{\mu}}{1-\exp \{-\alpha\}}+\frac{2 K C_{r}}{1-\exp \{\alpha-M\}}+\frac{2 K C_{r} \lambda}{1-\exp \{-(\alpha-M \beta)\}}\right] \|\xi-\bar{\xi}\|^{\beta} \\
\leq& \lambda\|\xi-\bar{\xi}\|^{\beta}.
\end{split}
\]
Now by the definition of $G(t, y),$ if $0<\|y-\bar{y}\|<1,$  then we conclude that
$$
\|G(t, y)-G(t, \bar{y})\| \leq\|y-\bar{y}\|+\lambda\|y-\bar{y}\|^{\beta} \leq(1+\lambda)\|y-\bar{y}\|^{\beta}=q\|y-\bar{y}\|^{\beta}.
$$
  Therefore, $G$ is H\"{o}lder continuous. This completes the proof of Step 1-2.
\end{proof}

\section*{Conflict of interest statement}
 The authors declare that there is no conflict of interests regarding the publication of this article.

\section*{Acknowledgements}
This work was jointly supported by Fondecyt project 1170466 and the National Natural
Science Foundation of China under Grant (11931016).


\begin{thebibliography} {}

\bibitem{Poincare}
H. Poincar\'{e}, Sur le probl\`{e}me des trois corps et les \'{e}quations de la dyanamique, {\em Acta Math.}, 13 (1890) 1--270.

\bibitem{Siegel}
C.  Siegel, Iteration of analytic functions, {\em Ann. of Math.}, 43 (1942) 607--612.

\bibitem{Yoccoz}
J. Yoccoz, Lin\'{e}arisation des germes de diff\'{e}omorphismes holomorphes $d(\mathbb{C},0)$, {\em C. R. Acad. Sci. Paris},
36 (1988) 55--58.

%\bibitem{Hale1}
%J. Hale, Ordinary Differential Equations, Pure Appl. Math., vol. 21, Wiley-Interscience, 1969.

%\bibitem{Hartman1}
%P. Hartman, A lemma in the theory of structural stability of differential equations, {\em Proc. Amer. Math. Soc.}, 11  (1960) 610--620.

\bibitem{Hartman} P. Hartman, On local homeomorphisms of Euclidean spaces, {\em Bol. Soc. Mat. Mexicana}, 5 (1960) 220--241.

\bibitem{Grobman}
D. Grobman, Homeomorphisms of systems of differential equations, {\em Dokl. Akad. Nauk SSSR}, 128 (1965) 880--881.

\bibitem{Palis}
J. Palis, On the local structure of hyperbolic points in Banach spaces, {\em An. Acad. Brasil. Ci\^{e}nc.}, 40 (1968)
263--266.

\bibitem{Pugh1}
C. Pugh, On a theorem of P. Hartman, {\em Amer. J. Math.}, 91 (1969) 363--367.

\bibitem{Lu2}
P. Bates, K. Lu, A Hartman-Grobman theorem for the Cahn-Hilliard and phase-field equations, {\em J. Dynam. Differential Equations}, 6 (1994) 101--145.

\bibitem{Lu1}
K. Lu, A Hartman-Grobman theorem for scalar reaction diffusion equations, {\em J. Differential Equations}, 93 (1991) 364--394.


\bibitem{Zgl-HG}P. Zgliczy\'nski, Topological shadowing and the Grobman-Hartman theorem, {\em Topol. Method. Nonl. An.},
  50 (2017), 757-785.

\bibitem{Hein-Pruss1}
M. Hein, J. Pr\"{u}ss, The Hartman-Grobman theorem for semilinear hyperbolic evolution equations, {\em J. Differential Equations}, 261 (2016) 4709--4727.

\bibitem{Palmer1}
K. Palmer, A generalization of Hartman's linearization theorem, {\em J. Math. Anal. Appl.}, 41 (1973) 753--758.

\bibitem{BDK-JDE}
L. Backes, D. Dragi\v{c}evi\'{c}, K. Palmer, Linearization and H\"{o}lder continuity for nonautonomous systems,
{\em J. Differential Equations}, 297 (2021) 536--574.

\bibitem{B-V1}
L. Barreira, C. Valls, A Grobman-Hartman theorem for nonuniformly hyperbolic dynamics, {\em J. Differential Equations}, 228 (2006) 285--310.

\bibitem{B-V2}
L. Barreira, C. Valls, A Grobman-Hartman theorem for general nonuniform exponential dichotomies, {\em J. Funct. Anal.}, 257 (2009) 1976--1993.

\bibitem{B-V3}
L. Barreira, C. Valls, Conjugacies between linear and nonlinear non-uniform contractions,
{\em Ergod. Theor. Dyn. Syst.}, 28 (2008) 1--19.


\bibitem{B-V4}
L. Barreira, C. Valls, Conjugacies for linear and nonlinear perturbations of nonuniform behavior, {\em J. Funct. Anal.}, 253 (2007) 324--358.


\bibitem{Huerta1}
\'{A}. Casta\~{n}eda, I. Huerta, Nonuniform almost reducibility of nonautonomous linear differential equations,
  {\em J. Math. Anal. Appl.}, 485 (2020) 123822.

\bibitem{Huerta2}
I. Huerta, Linearization of a nonautonomous unbounded system with nonuniform contraction: A spectral approach,
 {\em Discrete Contin. Dyn. Syst.}, 40 (2020) 5571--5590.

\bibitem {Jiang1}
L. Jiang, Generalized exponential dichotomy and global linearization, {\em J. Math. Anal. Appl.}, 315 (2006) 474--490.

\bibitem{Jiang2}
L. Jiang, Ordinary dichotomy and global linearization, {\em Nonlinear Anal.}, 70 (2009) 2722--2730.

\bibitem{Fenner-Pinto}
J. Fenner, M. Pinto, On a Hartman linearization theorem for a class of ODE with impulse effect, {\em Nonlinear Anal.}, 38 (1999) 307--325.

\bibitem{Xia1}
Y. Xia, X. Chen, V. Romanovski, On the linearization theorem of Fenner and Pinto, {\em J. Math. Anal. Appl.}, 400 (2013) 439--451.

\bibitem{Papa-A} G. Papaschinopoulos, A linearization result for a differential equation with
piecewise constant argument, {\it Analysis}, 16 (1996) 161--170.

\bibitem{Potzche1}
C. P\"{o}tzche, Topological decoupling, linearization and perturbation on inhomogeneous time scales, {\em J. Differential Equations}, 245 (2008) 1210--1242.

\bibitem{Reinfelds1}
A. Reinfelds, L. Sermone, Equivalence of nonlinear differential equations with impulse effect in Banach space, {\em Latv. Univ. Zint. Raksti.}, 577 (1992) 68--73.

\bibitem{Reinfelds-IJPAM}
A. Reinfelds, D. \v{S}teinberga, Dynamical equivalence of quasilinear equations, {\em Int. J. Pure Appl. Math.}, 98 (2015) 355-364.

\bibitem{Shi-Zhang1}
J. Shi, J. Zhang, The Principle of Classification for Differential Equations,  Science Press, Beijing, 2003 (in Chinese).

\bibitem{Xia-BSM} Y. Xia, R. Wang, K. Kou, D. O'Regan,
 On the linearization theorem for nonautonomous differential equations,
{\em Bull. Sci. Math.}, 139 (2015) 829--846.

\bibitem{Sternberg1}
S. Sternberg, Local $C^n$ transformations of the real line, {\em Duke Math. J.}, 24 (1957) 97--102.

\bibitem{Sternberg2}
S. Sternberg, Local contractions and a theorem of Poincar\'{e}, {\em Amer. J. Math.}, 79(1957) 809--824.

\bibitem{Sell1}
G. Sell, Smooth Linearization near a fixed point,  {\em Amer. J. Math.}, 107 (1985) 1035--1091.


\bibitem{Belitskii1}
G. Belitskii, Functional equations and the conjugacy of diffeomorphism of finite smoothness class, {\em Funct. Anal. Appl.}, 7 (1973) 268--277.

\bibitem{Belitskii3} G. Belitskii, V. Rayskinon, The Grobman-Hartman theorem in $\alpha$-H\"older class for Banach spaces, preprint.


\bibitem{Cuong} L. Cuong, T. Doan, S. Siegmund,  A Sternberg theorem for nonautonomous differential
equations, {\em J. Dynam. Differential Equations}, 31 (2019) 1279--1299.

\bibitem{ZWN-MZ}
D. Dragi\v{c}evi\'{c}, W. Zhang, W. Zhang, Smooth linearization of nonautonomous difference equations with a nonuniform dichotomy, {\em Math. Z.}, 292 (2019) 1175--1193.

\bibitem{DZZ-PLMS}
D. Dragi\v{c}evi\'{c}, W. Zhang, W. Zhang,  Smooth linearization of nonautonomous differential equations with a nonuniform dichotomy, {\em Proc. London Math. Soc.}, 121 (2020) 32--50.

\bibitem{ElBialy1}
M. ElBialy, Local contractions of Banach spaces and spectral gap conditions, {\em J. Funct. Anal.}, 182 (2001) 108--150.


\bibitem{R-S-1}
H. Rodrigues, J. Sol\`{a}-Morales, Linearization of class $C^{1}$ for contractions on Banach spaces, {\em J. Differential Equations}, 201 (2004) 351--382.

\bibitem{R-S-2}
H. Rodrigues, J. Sol\`{a}-Morales, Smooth linearization for a saddle on Banach spaces, {\em J. Dynam. Differential Equations}, 16 (2004) 767--793.

\bibitem{R-S-3}
H. Rodrigues, J. Sol\'a-Morales, Invertible Contractions and Asymptotically Stable ODE'S
that are not $C^1$-Linearizable, {\em J. Dynam. Differential Equations}, 18 (2006) 961--974.

%\bibitem{ZWN-JFA} W.M. Zhang, W.N. Zhang, $C^1$ linearization for planar contractions, {\em J. Funct. Anal.}, 260(2011), 2043-2063.

\bibitem{ZWN-JDE}
W. Zhang, W. Zhang, Sharpness for $C^1$ linearization of planar hyperbolic diffeomorphisms,  {\em J. Differential Equations},  257 (2014) 4470--4502.

\bibitem{ZWN-ETDS} W. Zhang, W. Zhang, $\alpha$-H\"older  linearization of hyperbolic diffeomorphisms with resonance, {\em Ergod. Theor. Dyn. Syst.}, 36 (2016) 310--334.


\bibitem{ZWN-MA} W. Zhang, W. Zhang, W. Jarczyk, Sharp regularity of linearization for $C^{1,1}$ hyperbolic diffeomorphisms, {\em  Math. Ann.},  358 (2014) 69--113.

\bibitem{ZWN-TAMS} W. Zhang, K. Lu, W. Zhang,  Differentiability of the conjugacy in  the Hartman-Grobman Theorem, {\em Trans. Amer. Math. Soc.}, 369 (2017)  4995--5030.

\bibitem{Tan-JDE}
B. Tan, $\sigma$-H\"{o}lder continuous linearization near hyperbolic fixed points in $\mathbb{R}^{n}$, {\em J. Differential Equations},
162 (2000) 251--269.

\bibitem{Shi}
J. Shi, K. Xiong, On Hartman's linearization theorem and Palmer's linearization theorem, {\it J. Math. Anal. Appl.}, 192 (1995) 813--832.



\bibitem{Pinto1}
R. Naulin, M. Pinto, Admissible perturbations of exponential dichotomy roughness. {\em Nonlinear Anal.}, 31 (1998) 559--571.

\bibitem{Pinto2}
A. Coronel, C. Maul\'en, M. Pinto, D. Sep\'ulveda,
Dichotomies and asymptotic equivalence in alternately advanced and delayed differential systems,
{\em J. Math. Anal. Appl.},  45 (2017) 1434--1458.

\bibitem{Pinto3}
M. Pinto, Perturbations of asymptotically stable differential systems
{\em Analysis}, 4 (1984) 161--175.

\bibitem{Pinto4}
M. Pinto, Asymptotic integration of a system resulting from the perturbation of an $h$--system,
{\em J. Math. Anal. Appl.}, 131 (1988) 194--216.

\bibitem{[4]} W. Coppel,
\emph{Dichotomies in Stability Theory},
Lect. Notes Math., vol. 629, Springer, Berlin/New York (1978).

\bibitem{Meiss-book}
J. Meiss, \emph{ Differential Dynamical Systems}, Society for Industrial and Applied Mathematics, (2007).

\bibitem{ZWN-Chinese}
W. Zhang, Generalized exponential dichotomies and invariant manifolds for differential equations, {\em Adv. Math. Chin.}, 22 (1993) 1-45.

\bibitem{Xia-Zhang} J. Chu, F. Liao, S. Siegmund, Y. Xia, W. Zhang, Nonuniform dichotomy spectrum and reducibility for nonautonomous equations, {\em Bull. Sci. Math.},139(2015),538-557.



%\bibitem{Pinto-A} M. Pinto, G. Robledo, A Grobman-Hartman theorem for differential equations with
%piecewise constant arguments of mixed type, {\em Z. Anal. Anwend.}, 37 (2018) 101--126.



%\bibitem{Robinson}
%C. Robinson, Dynamical systems: Stability, Symbolic Dynamics, and Chaos. Boca Raton, FL, CRC Press, (1999).


%\bibitem{Strien} S. van Strien,  Smooth linearization of hyperbolic fixed points without resonance conditions, {\em J.
%Differential Equations}, 85 (1990) 66--90.






%%%%%  \bibitem{Zgl-JDE} P. Zgliczy\'nski, Covering relations, cone conditions and the stable manifold theorem,
% {\em J. Differential Equations}, 246 (2009), 1774-1819.


%\bibitem{GHR-DCDS}
%M. Guysinsky, B. Hasselblatt, V. Rayskin, Differentiability of the Hartman-Grobman linearization,
%{\em Discrete Contin. Dyn. Syst.}, 9 (2003) 979--984.




\end{thebibliography}
\end{document}